\documentclass[a4paper,12pt]{amsart}
\usepackage{amssymb}

\usepackage{amsfonts,amssymb,amsmath,amsthm}
\usepackage{amsmath}
\usepackage{amsfonts}
\usepackage{amssymb}
\usepackage{amsthm}
\usepackage{latexsym}
\usepackage[active]{srcltx}
\usepackage{color}
\usepackage{comment}

 \newtheorem{thm}{Theorem}[section]
 \newtheorem{cor}[thm]{Corollary}
 \newtheorem{lem}[thm]{Lemma}
 \newtheorem{prop}[thm]{Proposition}
 \theoremstyle{definition}
 \newtheorem{defn}[thm]{Definition}
 \theoremstyle{remark}
 \newtheorem{rem}[thm]{Remark}
 
 \numberwithin{equation}{section}
\newtheorem*{Nt}{Notations}

\renewcommand{\H}{{\mathcal H}}
\def\C{\mathbb C}
\def\R{\mathbb R}
\def\S{{\mathcal S} }

\def\C{\mathbb C}
\def\R{\mathbb R}

\def\N{\mathbb N}
\def\al{\alpha}

\def\be{\beta}

\def\de{\delta}

\def\GA{\Gamma}

\def\ve{\varepsilon}

\def\va{\varphi}
\def\ta{\tau}

\def\va{\varphi}

\def\p{\mathfrak{p}}
\def\n{\mathfrak{n}}

\def\ve{\varepsilon}
\def\si{\sigma}

\def\ga{\gamma}
\def\ph{\phi}
\def\ch{\chi}
\def\ta{\tau}

\def\N{\mathbb{N}}
\def\Z{\mathbb{Z}}
\def\R{\mathbb{R}}
\def\C{\mathbb{C}}

\def\ol#1{\overline{#1}}
\def\nn{\nonumber}
\def\noop#1{\Vert #1\Vert_{\rm op}}

\def\R{{\mathbb R}}
\def\C{{\mathbb C}}
\def\N{{\mathbb N}}

\def\Z{{\mathbb Z}}

\def\B{{\mathcal B}}

\def\F{{\mathcal F}}

\def\H{{\mathcal H}}

\def\K{{\mathcal K}}

\def\O{{\mathcal O}}

\author[Ghofrane Kardi ]{ Ghofrane Kardi}

\address{Department of Mathematics, Faculty of Science of Gab\`es, university of Gab\`es; Tunisia}
\email{\sl ghofranekardi2018@gmail.com}
\title[The $C^\ast$-algebra  of nilpotent Lie group of seven-dimensional.]
 {The $C^\ast$-algebra  of nilpotent Lie group of seven-dimensional.}

\begin{document}
	\begin{abstract} $ $\\
In this paper, the $C^*$-algebra of the seven-dimensional un-decomposable  nilpotent Lie group is characterized explicitly for the first time(see \cite{chin}). Furthermore, the topology of its spectrum is described as a preparation for the analysis of its $C^*$-algebra. Then, the operator-valued Fourier transform is employed to translate the given C*-algebra into the algebra of bounded operator fields through its spectrum. We find the conditions satisfied by the image of the goal to characterize them by these conditions, namely the "norm controlled dual limit" (NCDL)-conditions (see \cite{Lud-Reg}). The methods used for the nilpotent Lie groups are different for each. We consider the co-adjoint orbits and for our case, we use the Kirillov theory.
	\end{abstract}

	

	\maketitle
	\tableofcontents

	
		\section{Introduction }\label{sec1}
With the aim of continuing the description of the $C^*$-algebras of low-dimensional nilpotent Lie groups as algebras of operator fields defined over their spectra, we illustrate in this article that the $C^*$-algebra of $7$-dimensional nilpotent Lie group, with the group $G_{6,17}$ as the quotient group modulo the centre, exhibit norm controlled dual
limits. (The notation of $G_{6,17}$ is given by Nielsen in \cite{Nie}).
To understand these $C^*$-algebra, the Fourier transform plays a crucial role. So, let 
$\mathcal{X}$ be a $C^*$-algebra and $\widehat{\mathcal{X}}$ its unitary dual space, which is the set of all equivalence classes of the irreducible unitary representation of $\mathcal{X}$. For every element $x\in \mathcal{X}$, the Fourier transform $\mathcal{F}$ has the following definition:\\
select in each equivalence class $\lambda\in \widehat{\mathcal{X}}$ a representation $(\phi_{\lambda},\mathcal{H}_{\lambda})$ and define:
$$\begin{matrix}
	\mathcal{F}:&\mathcal{X} &\rightarrow& \ell^{\infty}(\widehat{\mathcal{X}})& \ & \ & \ \\
	\ & x & \mapsto & \mathcal{F}(x)=\widehat{x}:& \widehat{\mathcal{X}}& \rightarrow & \mathcal{B}(\mathcal{H}_{\lambda})\\
	\ & \ &\ &\ & \lambda& \mapsto& \widehat{x}(\lambda)= \phi_{\lambda}(x)
\end{matrix}$$
where $\mathcal{B}(\mathcal{H}_{\lambda})$ is the $C^*$-algebra of bounded linear operators on $\mathcal{H}_{\lambda}$ and $\ell^{\infty}(\widehat{\mathcal{X}})$ is the algebra of all bounded operator fields over $\widehat{\mathcal{X}}$ i. e.
$$\ell^{\infty}(\widehat{\mathcal{X}})=\{\psi= (\psi(\phi)\in \mathcal{B}(\mathcal{H}_{\phi}))_{\phi \in \widehat{\mathcal{X}} }, ||\psi||_{\infty}<+\infty\}$$
where $||\psi||_{\infty}=\underset{\psi \in \widehat{\mathcal{X}}}{\sup}||\psi(\phi)||_{\text{op}}$. Also, it is well-known that the mapping $\mathcal{F}$ is an isometric  $*$-homomorphism from $\mathcal{X}$ into $\ell^{\infty}(\widehat{\mathcal{X}})$.\\
Now,  we denote by $dn$ the Haar measure of the locally compact group $N$,  we have that the $C^*$-algebra of $N$ is defined as the completion of the convolution algebra $L^1(N)$ concerning the $C^*$-norm of $L^1(N)$ i. e.
$$C^*(N)=\overline{L^1(N,dn)}^{||.||_{C^*(N)}},\ \text{where}\ ||f||_{C^*(N)}=\underset{[\phi]\in \widehat{N}}{\sup}||\phi(f)||_{\text{op}}$$ where $\widehat{N}$ is the spectrum of $N$. Also, we can write any irreducible unitary representation $(\tilde{\phi},\mathcal{H})$ of $L^1(N)$ as a unique integral in the way shown above, and thus we get an irreducible unitary representation $(\phi,\mathcal{H})$ of $N$. This gives the famous result, shown in \cite{Dixmier} which proves that the spectrum of $C^*(N)$, corresponds to the spectrum of $N$ i. e. $\widehat{C^*(N)}=\widehat{N}.$\\
The spectrum structure in this work is more complicated than dimensions $5$ and $6$, and it is necessary 
to understand new spectral topology phenomena. Due to Kirillov's orbit picture of the spectrum of a connected and simply connected nilpotent Lie group, we get a description of the structure of the space of co-adjoint orbits of these groups.
But orbit theory is only an algorithm, it gives no details about the result of computations.
The topology of the orbit space or the behavior of the operators $\phi(F)$, $F\in C^*(N)$ as $\phi$ varies in the spectrum is different for each of these groups and must be studied on a case-by-case basis. Then,
For certain classes of Lie groups, the description of the  $C^*$-algebras is already known. We will  use a result of   J. Ludwig and H. Regeiba  (see \cite{Lud-Reg}, Theorem $3.5$)  to characterize the
$C^*$-algebras mentioned in the next section.\\
The structure of the $C^*$-algebras is already known for certain classes of nilpotents Lie groups , such as the Heisenberg  and the threadlike Lie groups in \cite{Lud-Tur},  the $ax+b$-like groups, that can be found in \cite{Lin-Lud}, and the $C^*$-algebra of the two step nilpotent Lie group in \cite{Gun-Lud}. Furthermore the $C^*$-algebra of the $5$-dimensional and some of the $6$-dimensional Lie groups  have been determined  in \cite{Lud-Reg2}  and \cite{Lud-Reg}.\\
In another way, other mathematicians characterized the $C^*$-algebra of montion group $SO(n)\ltimes\R^n$ (see \cite{Ell-Lud}), moreover the Heisenberg  montion groups $\mathbb{T}^n \ltimes \mathbb{H}_n$ and the semi-direct product $K \ltimes A$ (see\cite{Rej-Lud} and \cite{Hed-Jea}). In the class of variable groups we have as example the $C^*$-algebra of the variable Mautner groups in \cite{Regeiba}.\\
In the present paper the group $C^*$-algebras of the nilpotent
Lie group of $7$-dimensions will be described very explicitly. In the second section the definition of a $C^*$-algebra with NCDL is given. The conditions required by this definition characterize group $C^*$-algebras. Then, in the third section, we will present the nilpotent Lie groups of $7$-dimensions.
\section{Preliminary}
\begin{defn}\cite{Lud-Reg}
	We say that a $C^*$-algebra $\mathcal{X}$ is a $C^*$-algebra with NCDL if the following conditions are satisfied:
	\begin{enumerate}
		\item Stratification of the spectrum:
		\begin{itemize}
			\item[a.] There is an increasing family with  finite dimension $F_0\subset F \subset ...\subset F_r =\widehat{\mathcal{X}}$ of closed subsets of the unitary dual space
 $\widehat{\mathcal{X}}$ of $\mathcal{X}$ in such a way that for $i\in \{1,..,r\}$ the subsets $\Gamma_0=F_0$ and
			$\Gamma_i:=F_i\ F_{i-1}$ is Hausdorff in their relative topology and such that $F_0$ is the set of all characters of $\mathcal{X}$.
			\item [b.] For every $i\in \{1,...,r\}$, there exists a Hilbert space $\mathcal{H}_i$ and for all $\lambda\in \mathcal{H}_i$, there exists a concrete realisation, on the Hilbert space $\mathcal{H}_i$, $(\phi_{\lambda},\mathcal{H}_{i})$ of $\lambda$.
		\end{itemize}
		\item $CCR$ $C^*$-algebra:\\
		$\mathcal{X}$ is called a separable $CCR$ $C^*$-algebra, if it  is a separable $C^*$-algebra and the
		image of every irreducible representation $(\phi,\mathcal{H})$ of $\mathcal{X}$ is contained in the algebra of compact
		operators $\mathcal{K}(\mathcal{H})$ (which implies that the image equals to $\mathcal{K}(\mathcal{H})$).
		\item The passage between layers:
		We take $x\in \mathcal{X}$.
		\begin{itemize}
			\item[a.] On the every sets $\Gamma_i$,  the mappings $\lambda\rightarrow \mathcal{F}(x)(\lambda)$ are continuous with respect to the norm.
			\item [b.] For any $i\in \{1,..,r\}$ and for any converging sequence contained in $\GA_i$ with limit set
			outside $\GA_i$ (thus in $F_{i-1}$), there is a properly converging sub-sequence $\lambda=(\lambda_k)_{k\in \mathbb{N}}$,
			as well as a constant $M>0$ and for every $k\in \mathbb{N}$ an involutive linear mapping
			$\tilde{\sigma_k}: CB(F_{i-1})\rightarrow \mathcal{B}(\mathcal{H}_i)$, which is bounded by $M||.||_{F_{i-1}}$ (uniformly in $k$), such
			that $$ \lim_{k\rightarrow \infty}||\mathcal{F}(x)(\lambda_k)-\tilde{\sigma_k}(\mathcal{F}(x)_{|F_{i-1}})||_{\text{op}}=0$$
			where, $CB(F_{i-1})$ is the $*$-algebra of all the uniformly bounded fields of operators
			$(\omega(\lambda)\in \mathcal{B}(\mathcal{H}_{f_{j}}))_{\lambda \in \Gamma_{j}, j=0,...,i-1}$, which are operator norm continuous on the subsets $\GA_{j}$
			for every $j\in \{0,...,i-1\},$ provided with the infinity-norm
			$$||\omega||_{F_{i-1}}:=\sup_{\lambda\in F_{i-1}}||\omega(\lambda)||_{\text{op}}.$$
			
		\end{itemize}
	\end{enumerate}
\end{defn}
Now, let $\n$ be the nilpotent Lie algebra of $N$ for all $f \in \n^*$, consider the skew-bi-linear form
$$B_f(X,Y)=\langle f,[X,Y]\rangle$$  where $\langle\cdot,\cdot\rangle$ the scalar product on $\n$.

Let
$$\n(f)=\left\{X\in\n|\ \langle f,[X,\n]\rangle=\{0\}\right\}$$
be the radical of $B_f$ and the stabilizer of the linear functional $f.$

\begin{defn}
	A sub-algebra $\p$ of $\n$, that is subordinated to $f$ (i.e $\langle f,[\p,\p]\rangle=\{0\}$) and that has the dimension
	$$\dim{\p}=\frac{1}{2}\left(\dim{\n}+\dim{\n(f)}\right),$$
	which means that $\p$ is maximal isotropic for $B_f,$ is called a polarization of $f.$
	
	If $\p\subset\n$ is any sub-algebra of $\n$ which is subordinated to $f$, the linear functional $f$ defines a unitary character $\chi_f$ of $P=\exp(\p):$
	$$\chi_f(A)=e^{-2i\pi \langle f,\log(A)\rangle},\ \forall A\in P.$$
\end{defn}
	\begin{defn}
		For all $X\in N$ the co-adjoint action $\text{Ad}^*$ can be defined as follows:
		\begin{eqnarray*}
			\begin{array}{ccccccc}
				\text{Ad}^*(X): & \n^* &\longrightarrow&\n^*&&& \\
				&f&\longmapsto&\text{Ad}^*(X)f&:\n&\longrightarrow&\n\\
				& & & &B&\longmapsto&\text{Ad}^*(X)f(B)=f(\text{Ad}(X^{-1})B)
			\end{array}
		\end{eqnarray*}
	\end{defn}

Let $K$ be closed subgroup of $N$ and define
\begin{eqnarray*}
	L^2(N/K,\chi_f)&=&\Big\{\zeta:N\longrightarrow\C|\zeta\text{ measurable, } \zeta(nm)=\chi_f(m^{-1})\zeta(n), \\
	& & \forall n\in N,\ \forall m\in K, ||\zeta||_2^2=\int_{N/K}|\zeta(n)|^2d\dot{n}<\infty\Big\},
\end{eqnarray*}
where $d\dot n$ is an invariant measure on $N/K$ which always exists for nilpotent $N$.
$L^2(N/K,\chi_f)$ is a Hilbert space and one can define the unitary induced representation:
$$\text{ind}_K^{N}\chi_f(x)\zeta(y)=\zeta(x^{-1}y),\ \forall x,y\in N,\ \forall\zeta\in L^2(N/K,\chi_f).$$
If the associated Lie algebra $\mathfrak{k}$ of $K$ is a polarization, then the induced
representation is also irreducible.
Now, by the Kirillov theory ( \cite{Cor-Gre} and \cite{Lep-Lud}), we have for every representation class
$\sigma\in\widehat{N}$, there exists
an element $f\in\n^*$ and a polarization $\p$ of $f$ in $\n$ such that $\sigma=[\text{ind}_P^{N}\chi_f],$  where $P := \exp(\p).$
Moreover, if $f,f'\in\n_n^*$ are located in the same co-adjoint orbit $\O \in\n^*/N$,$\p$ be a polarization in $f$, and let $\p'$ be a polarization in $f'$, the induced representations $\text{ind}_P^{N}\chi_f$ and $\text{ind}_{P'}^{N}\chi_{f'}$ are
equivalent and thus, the Kirillov map which goes from the co-adjoint orbit space $\n^*/N$ to the spectrum $\widehat{N}$ of $N$
$$\begin{array}{cccc}
	\K: & \n^*/N &\longrightarrow&\widehat{N} \\
	&\text{Ad}^*(N)f&\longmapsto&[\text{ind}_P^{N}\chi_f]\\
\end{array}$$
is a homeomorphism. Therefore,
$$\n^*/N\cong\widehat{N}$$
as topological spaces.
We can parametrize the orbit space $\n^*/N$ in the following way. We have a decomposition
$$\n^*/N=\underset{i}{\dot{\bigcup}}\Gamma_i,$$
For $1\leq i\leq r$ let $f\in\Gamma_i$ let us describe the representation $\phi_f$ explicitly. The Hilbert space $\H_f$ of the representation $\phi_f$ is the space $L^2(N/P_f,\chi_f).$
We identify the group $N$ by $N/P_f P_f$ as topological products. Hence for $n=xy\in  N/P_fP_f,\ \zeta\in\H_f$ and $z\in N/P_f$ we have
\begin{eqnarray} \label{pil}
	\phi_f(n)\xi(z) \nn &=&\zeta(n^{-1}z)\\ \nn
	&=&\zeta(y^{-1}x^{-1}z)\\ \nn
	&=&\zeta(x^{-1}z(x^{-1}z)^{-1}y^{-1}x^{-1}z)\\ \nn
	&=&\chi_f((x^{-1}z)^{-1}yx^{-1}z)\zeta(x^{-1}z)\\
	&=&e^{-2i\pi \langle\text{Ad}^*(x^{-1}z)(f),\log(y)\rangle}\zeta(x^{-1}z).
\end{eqnarray}
\section{The $C^*$-algebra of the group $N_{7}.$}
\subsection{Description of the group $N_{7}$.}$ $
In this section, we describe the co-adjoint orbits and the irreducible representations of the group $N_{7}$,  $(N_{7}=(137A_{1})$ This notation is obtained by the classification of algebras by its upper central series dimensions see \cite{chin} $)$.\\ Let $\textbf{n}_{7}$  be the nilpotent Lie algebra 
with basis 
    $B=\{ X_{1}, X_{2}, X_{3},X_{4},X_{5},X_{6},X_{7}\}$
and  non trivial commutators:
\begin{align*}
	&[X_{1},X_{3}]=X_{5}&&[X_{1},X_{4}]=X_{6}&&[X_{1},X_{5}]=X_{7}\\ \nonumber&[X_{2},X_{3}]=-X_{6} &&[X_{2}, X_{4}]= X_{5}&& [X_{2},X_{6}]=X_{7} 
	&&\ 
\end{align*} 
For all $X$ $\in$ $\textbf{n}_{7}$, $X=x_{1}X_{1}+x_{2}X_{2}+x_{3}X_{3}+x_{4}X_{4}+x_{5}X_{5}+x_{6}X_{6}+x_{7}X_{7}$, such that $x_{i}\in \mathbb{R}$ and  $ i \in \{1,2,3,4,5,6,7\}$. So we can identify $\textbf{n}_{7}$ by $\mathbb{R}^{7}$.\\
The largest abelian ideal of $\textbf{n}_{7}$ is the sub set spanned by $<X_3,X_{4},X_{5},X_{6},X_{7}>$. Then with Python, we obtained the following multiplication which equipped the Lie group $N_{7}$: 

\begin{align*}
	&\nonumber(x_{1},x_{2},x_{3},x_{4},x_{5},x_{6},x_{7}).(x_{1}^{'},x_{2}^{'},x_{3}^{'},x_{4}^{'},x_{5}^{'},x_{6}^{'},x_{7}^{'})\nonumber\\&=\nonumber(x_{1}+x_{1}^{'},x_{2}+x_{2}^{'},x_{3}+x_{3}^{'},x_{4}+x_{4}^{'},x_{5}+x_{5}^{'}-x_{1}^{'}x_{3}+\frac{x_{2}x_{4}^{'}}{2}- \frac{x_{2}^{'}x_{4}}{2},
	x_{6}+x_{6}^{'}\nonumber\\&-x_{1}^{'}x_{4}+\frac{x_{2}^{'}x_3}{2}-\frac{x_{2}x_{3}^{'}}{2},x_{7}+x_{7}^{'}-x_{1}^{'}x_{5}+\frac{(x_{1}^{'})^2x_{3}}{2}-\frac{x_{2}^{2}x_{3}^{'}}{12}+\frac{x_{2}x_{6}^{'}}{2}-\frac{x_{2}^{'}x_{6}}{2}\nonumber\\&-\frac{(x_{2}^{'})^2x_{3}}{12}+\frac{x_{1}^{'}x_2^{'}x_4}{2}+\frac{x_2x_2^{'}x_3}{12}-\frac{x_2x_2^{'}x_3^{'}}{12}).
\end{align*}
where,
$ (x_{1},x_{2},x_{3},x_{4},x_{5},x_{6},x_{7}),(x_{1}^{'},x_{2}^{'},x_{3}^{'},x_{4}^{'},x_{5}^{'},x_{6}^{'},x_{7}^{'})\in N_{7}$.
\begin{align*}
	&\nonumber Ad^{*}((x_{1},x_{2},x_{3},x_{4},x_{5},x_{6},x_{7}))(f_{1},f_{2},f_{3},f_{4},f_{5},f_{6},f_{7})\nonumber\\&= \nonumber(f_1-f_5x_3-f_6x_4-f_7x_5-f_7x_1x_3-f_7\frac{x_2x_4}{2},f_2-f_5x_4+f_6x_3-f_7x_6\nonumber\\&-f_7x_1x_4+f_7\frac{x_2x_3}{2},f_3+f_5x_1-f_6x_2+f_7\frac{x_1^2}{2}-f_7\frac{x_2^2}{2},f_4+f_5x_2+f_6x_1\nonumber\\&+f_7x_1x_2,f_5+f_7x_1,f_6+f_7x_2,f_7)	
\end{align*}
where $(x_{1},x_{2},x_{3},x_{4},x_{5},x_{6},x_{7})\in \n_{7}$ and $(f_{1},f_{2},f_{3},f_{4},f_{5},f_{6},f_{7})\in \n_{7}^{*}$.
\subsection{Description of the co-adjoint orbits of the group $N_{7}$.}$ $\\
We will now describe the different layers of the co-adjoint orbit space $\n_{7}^*/N_{7}$. For all element 
$(f_{1},f_{2},f_{3},f_{4},f_{5},f_{6},f_{7})\in n_{7}^{*}$. We distinguish three layers of co-adjoint orbits, which we will be characterized in the following paragraph: (To simplify the notation, instead of writing $f_{f_i}$, we write $f_{i}$ for $i\in  \{1,...,7\} $).
\begin{itemize}
	\item[1.] As usual, we start with the \textbf{generic layer}. Let 
	$f_7\neq0$, the co-adjoint orbit can be parameterized by:	
	$\O_{(f_3,f_4,f_7)}:=\{(x_1,x_2,f_3+\frac{x_5^2}{2f_7}-\frac{x_6^2}{2f_7},f_4+\frac{x_6x_5}{f_7},x_5,x_6,f_7), x_1,x_2,x_5,x_6\in \R^4\}$ and the polarization at $(f_4,f_7)$ is the sub-algebra $\p_{(f_4,f_7)}:=\text{span}<X_3,X_4,X_5,X_6,X_7>$. 
	We denoted by $\GA_2$ the orbit space of this layer, where
	\begin{eqnarray*}
		\GA_2:=\text{span}\{f=(f_4,f_7), (f_4,f_7)\in  \R \times \R^* \}.
	\end{eqnarray*}
\item[2.] Now, let's consider the case where $f_7=0$, $f_6$ and $f_5$ $\neq0$. In this layer, we remark that $N_{7}/\exp(Z)=G_{6,17}$ where $Z=<X_7>$ is the center of $\n_{7,2}$ and  $G_{6,17}$ is a nilpotent Lie group which is well detailed in the thesis of H. Regaiba \cite{Reg.the}. To properly parameterize the orbit, we need to take a new basis:
$\mathcal{B}_{(f_5,f_6)}=\{X_1^{f_5,f_6}=-f_5X_1+f_6X_2,X_2^{f_5,f_6}=-f_6X_1-f_5X_2,X_3^{f_5,f_6}=-f_6X_3+f_5X_4,X_4^{f_5,f_6}=f_5X_3-f_6X_4,X_5,X_6,X_7\}$. Then, in
the dual basis $\mathcal{B}_{(f_5,f_6)}^*$ of the basis $\mathcal{B}_{(f_5,f_6)}$, the orbit $\O_{(f_5,f_6)}$ is given by:\\
$\O_{(f_5,f_6)}=\{x_1(X_1^{f_5,f_6})^*+x_2(X_2^{f_5,f_6})^*+x_3(X_3^{f_5,f_6})^*+x_4(X_4^{f_5,f_6})^*+f_5X_5^*+f_6X_6^*, \text{such that }\ (x_1,x_2,x_3,x_4)\in \R^4\}$\\
$\bullet$ If $f_6\neq 0$ and $f_5=0$:
\begin{align*}
	\mathcal{O}_{(f_6)}:=\{(x_1,x_2,x_3,x_4,0,f_6,0), x_1,x_2,x_3,x_4\in  \R^4 \}.
\end{align*}
The stabilizer of $(f_6)$ is the set:
\begin{eqnarray*}
	\n_{7}((f_6)):=\text{span}<X_5,X_6,X_7>.
\end{eqnarray*}
The polarization $\p_{(f_6)}:=\text{span}<X_3,X_4,X_5,X_6,X_7>$.\\ 
$\bullet$ If $f_6=0$ and $f_5\neq0$:
\begin{eqnarray*}
	\O_{(f_5)}=\{(x_1,x_2,x_3,x_4,f_5,0,0), x_1,x_2,x_3,x_4\in \R^4 \}.
\end{eqnarray*}
The stabilizer of $(f_5)$ is the set:
\begin{eqnarray*}
	\n_{7}((f_5)):=\text{span}<X_5,X_6,X_7>.
\end{eqnarray*}
The polarization $\p_{(f_5)}:=\text{span}<X_3,X_4,X_5,X_6,X_7>$.\\ 
$\bullet$ Similarly in the case if $f_6$ and $f_5$ $\neq 0$.\\
So we conclude that
\begin{eqnarray*}
	\O_{(f_5,f_6)}=\{(x_1,x_2,x_3,x_4,f_5,f_6,0), x_1,x_2,x_3,x_4\in \R^4 \}.
\end{eqnarray*}
We denoted by $\GA_1$ the orbit space of this layer, where
\begin{eqnarray*}
	\GA_1:=\text{span}\{f=(f_5,f_6), (f_5,f_6)\in  \R^* \times \R^* \}.
\end{eqnarray*}
The polarization $\p_{(f_5,f_6)}:=\text{span}<X_3,X_4,X_5,X_6,X_7>$. 
\item[3.] Finally, if $f_7=f_6=f_5=0$.
Denoted by 
$$\GA_0=(\n_{7}^{*}/N_{7})_{\text{char}}\simeq\R^4.$$
is the collection of all characters $$f=(f_1,f_2,f_3,f_4) = f_1X_1^*+f_2X_2^*+f_3X_3^*+f_4X_4^*,\  \forall \ f_1,f_2,f_3,f_4\in\R.$$

Their orbits are the one-point sets. Hence also the dual space
$\widehat{N_{7}}$, into the disjoint union
\begin{eqnarray*}
	\n_{7}^*/N_{7}=\GA_2\dot\cup \GA_1 \dot\cup \GA_0.
\end{eqnarray*}
\end{itemize}
\subsection{The topology of the dual space of  the group $N_{7}$.}$ $\\
In this section, we denote by 
$\ol{\O}$ the sequence $(\O_k)_k\subset\n^*_{7}/N_{7}$ and by $L(\ol{\O})$ its limit set  in $\n_{7}^*/N_{7}$.
\begin{thm}
	The subset $\GA_{2}$ has separated topology as sub-set of $\n_{7}^*/N_{7}$. In particular $\GA_{2}$ is homeomorphic to $\R^2 \times \R^*$.
\end{thm}
\begin{proof}
	Let $\overline{\O}=(\O_k=\O_{(f_3^k,f_4^k,f_7^k)})_k$ be a sequence in $\GA_{2}$. If $\overline{\O}$ admits a limit point $\O_{(f_3,f_4,f_7)}$ in $\GA_{2}$ then $f_7^k$ converges to $f_7$ and we have for every $k\in \N$ an element: $$(x_1^k,x_2^k,f_3^k+\frac{(x_5^k)^2-(x_6^k)^2}{2f_7^k},f_4^k+\frac{x_5^kx_6^k}{f_7^k},x_5^k,x_6^k,f_7^k)\in \O_k$$
	which converges to $(0,0,f_3,f_4,0,0,f_7)$. Hence $\lim_{k}f_3^k=f_3$ and $\lim_{k}f_4^k=f_4$, since $f_7\neq0$. If on the other hand the sequence $(f_3^k,f_4^k,f_7^k)\subset  \R^2 \times \R^*$ converges to $(f_3,f_4,f_7)\in \R^2 \times \R^*$. Then, $\lim_{k}\O_k=\O_{(f_3,f_4,f_7)}$ in $\n_{7}^*/N_{7}$. 
\end{proof}
\begin{thm}\label{2.2}
	Let $\overline{\O}=(\O_{(f_3^k,f_4^k,f_7^k)_k})\subset \GA_{2}$ be an orbit sequence, such that $\underset{k\rightarrow \infty}{\lim}f_7^k=0$.  $\overline{\O}$ has a limit(converges) if and only if $\underset{k\rightarrow \infty}{\lim}f_4^kf_7^k:=c_1$ and $\underset{k\rightarrow \infty}{\lim}2f_3^kf_7^k:=c_2$ exist in $\R$, and 
	\begin{itemize}
		\item if $c_1\neq 0$: $L(\overline{\O})=\{\O_{(f_5,-\frac{c_1}{f_5})}\in \GA_1; f_5\neq0 \}$
		\item if $c_1=0$: $L(\O)=\{\O_{(f_5,f_6)}\in \GA_1; f_5.f_6=0 \}\cup\GA_0.$
	\end{itemize}
\end{thm}
	
\begin{proof}
	Take $(\O_k=\O_{(f_3^k,f_4^k,f_7^k)})_k\subset \GA_{2}$, suppose that the sequence $(\O_k)_k$ converges to some $\O_f$ for  some $f=(f_1,f_2,f_3,f_4,f_5,f_6,0)\in \n_{7}^*$. Take for every $k\in \N$ the element $$f_k=(x_1^k,x_2^k,f_3^k+\frac{(x_5^k)^2-(x_6^k)^2}{2f_7^k},\frac{f_7^kf_4^k+x_5^kx_6^k}{f_7^k},x_5^k,x_6^k,f_7^k)\in \O_k$$	such that $\lim_{k}f_k=f$. We have that $\lim_{k}x_5^k=f_5$ and $\lim_{k}x_6^k=f_6$. Furthermore, since $\lim_{k} \frac{f_7^kf_4^k+x_5^kx_6^k}{f_7^k}=f_4$. We must have $\lim_{k} f_7^kf_4^k+x_5^kx_6^k=0$ and $\lim_{k}2f_3^kf_7^k+(x_5^k)^2-(x_6^k)^2=0$ i. e. $\lim_{k}f_7^kf_4^k=-f_6f_5:=c_1$ and $\lim_{k}2f_7^kf_3^k=\lim_{k}(x_6^k)^2-(x_5^k)^2=f_6^2-f_5^2:=c_2=\frac{c_1^2}{f_5^2}-f_5^2(f_5\neq0)$.\\ Suppose now that $\lim_{k} f_7^kf_4^k=c_1$, let $f=(f_1,f_2,f_3,f_4,f_5,f_6,0)$ be any element of $\n_{7}^*\cap(X_7)^{\bot}$ such that $f_6f_5=-c_1$. 
	\begin{itemize}
		\item If $c_1\neq0, (f_5\neq0)$, take  $f_k=(f_1,f_2,f_3,f_4,f_5,\frac{f_4f_7^k-f_7^kf_4^k}{f_5},f_7^k)\in \O_k$. Then $(f_k)_k$ converges to $f$.
		\item If $c_1=0$, $(f_5=0\ \text{or}\ f_6=0)$. Let take 
		\begin{eqnarray*}
		f_k&=&\Big(f_1,f_2,f_3,f_4,\sqrt{|f_7^kf_4-f_7^kf_4^k|},\text{sign}(f_7^kf_4-f_7^kf_4^k)\sqrt{|f_7^kf_4-f_7^kf_4^k|},\\ && f_7^k\Big)\in \O_k.	
		\end{eqnarray*}
		
	 Hence $\lim_{k}f_k=f$. This shows that $L(\overline{\O})$ is the set announced in the theorem.
		
	\end{itemize}	
\end{proof}	
\begin{thm}
	The subset $\GA_1$ has separated topology as sub-set of $\n_{7}^*/N_{7}$. In particular $\GA_1$ is isomorphic  to $\R^* \times \R^*$. 
\end{thm}
\begin{proof}
	Let $\overline{\O}=(\O_{(f_5^k,f_6^k)})_k$ be a sequence in $\GA_1$. If $\overline{\O}$ admits a point $\O_{(f_5,f_6)}$ in $\GA_1$ then $f_7^k$ and $f_6^k$ converge to $f_7$ and $f_6$ 
	respectively, and we have for every $k\in \N$ an element:
	$$(x_1^k,x_2^k,x_3^k,x_4^k,f_5^k,f_6^k,0)\in \O_k$$ which converges to $(0,0,0,0,f_5,f_6,0)$. Hence $\lim_{k} f_5^k=f_5\neq0$ and   $\lim_{k}f_6^k=f_6\neq0$. If in the other hand the sequence $(f_5^k,f_6^k)\subset \R^* \times \R^*$ converges to $(f_5,f_6)\in \R^* \times \R^*$. Then, $\lim_{k}\O_k=\O_{(f_5,f_6)}$ in $\n_{7}^*/N_{7}$.	
\end{proof}
\begin{thm}
	Let $\overline{\O}=(\O_{(f_5^k,f_6^k)})_k $  be an orbit sequence in $\GA_1$, such that $ \underset{k\rightarrow \infty}{\lim}f_5^k=0$ and $\underset{k\rightarrow \infty}{\lim}f_6^k=0$. Then, $L(\overline{\O})=\GA_0$.
\end{thm}
\begin{proof}
	We use the same idea of Proof of Theorem \ref{2.2}.
\end{proof}	
\subsection{The Fourier transform.}$ $\\	
In this section, we will describe explicitly the representation $\phi_f$ for all $f\in \GA_i$ $(0\leq i\leq 2)$. We have $P_f=\exp(\p_f)$ where $\p_f$ is the polarisation at $f$  and 
$$ind^{N_{7}}_{P_f}\chi_f(z)\zeta(t)=\zeta(z^{-1}t); \forall z,t \in N_{7},\zeta \in L^2(N_{7}/P_f, \chi_f)$$
the induced representation. 
Now, we identify the group $N_{7}$ by $N_{7}/P_f.P_f$ as topological products. Hence $\forall n=x.y\in N_{7}/P_f.P_f, \zeta \in \mathcal{H}_{f}$ and $z\in N_{7}/P_f$ we have by equation (\ref{pil})
\begin{eqnarray*}
	\pi_f(n)\xi(z)
	&=&e^{-2i\pi \langle\text{Ad}^*(x^{-1}z)(f),y\rangle}\xi(x^{-1}z).
\end{eqnarray*}  
Then, for $F\in L^1(N_{7}),\ \xi\in \H_f$ and $z\in N_{7}/P_f$ we have 
\begin{eqnarray*}
	& & \pi_f(F)\xi(z)\\&=&\int_{N_{7}}F(n)\pi_f(n)\xi(z)dn\\
	&=&\int_{N_{7}}F(x,y)e^{-2i\pi \langle\text{Ad}^*(x^{-1}z)(f),y\rangle}\xi(x^{-1}z)dxdy\\
	&=&\int_{N_{7}} F(zx^{-1},y)e^{-2i\pi \langle\text{Ad}^*(x)(f),y\rangle}\xi(x)dxdy\\
	&=&\int_{N_{7}/P_f}\widehat F^{P_f}(zx^{-1},\text{Ad}^*(x)(f)|_{P_f})\xi(x)dx.
\end{eqnarray*}
More explicitly, we have:
\begin{itemize}
	\item Let $P_{(f_3,f_4,f_7)}=\exp(\p_{(f_3,f_4,f_7)})$ for $n\in N_{7}, \xi \in\mathcal{H}_f,$\\ $X,X^{'}\in N_{7}/P_{(f_3,f_4,f_7)}$ we have
	\begin{eqnarray*}
		\pi_f(n)\xi(X^{'})=e^{-2i\pi<Ad^*(X^{-1}X^{'})(f),log(Y)>}\xi(X^{-1}X^{'}).
	\end{eqnarray*}
	And for $F\in L^{1}(N_{7}),\xi \in L^2(\R^2),$ we have 
	\begin{eqnarray*}
		\pi_f(F)\xi(X^{'})=\int_{\R^2}\widehat{F}^{P_{(f_4,f_7)}}(X^{'}X^{-1},Ad^{*}(X)(f)_{|P_{(f_4,f_7)}})\xi(X)dX.
	\end{eqnarray*}
	\item Let $P_{(f_5,f_6)}=\exp(\p_{(f_5,f_6)})$ for $n\in N_{7}, \xi \in\mathcal{H}_f,X,X^{'}\in N_{7}/P_{(f_5,f_6)}$ we have
	\begin{eqnarray*}
		\pi_f(n)\xi(X^{'})=e^{-2i\pi<Ad^*(X^{-1}X^{'})(f),log(Y)>}\xi(X^{-1}X^{'}).
	\end{eqnarray*}
	And for $F\in L^{1}(N_{7}),\xi \in L^2(\R^2),$ we have 
	\begin{eqnarray*}
		\pi_f(F)\xi(X^{'})=\int_{\R^2}\widehat{F}^{P_{(f_5,f_6)}}(X^{'}X^{-1},Ad^{*}(X)(f)_{|P_{(f_5,f_6)}})\xi(X)dX.
	\end{eqnarray*}
	\item We have that  the group $N_{7}$ is  nilpotent, thanks to Lie theorem, so all the irreducible  representations with  finite-dimensional of $N_{7}$ are one dimensional. Every one-dimensional representation is a unitary character
	$\chi_{(f_1,f_2,f_3,f_4)},$\\ $ (f_1,f_2,f_3,f_4)\in\R^4,$ of $N_{7}$ which is defined by the following formula:
	\begin{eqnarray*}
		\chi_{(f_1,f_2,f_3,f_4)}(x_1,x_2,x_3,x_4,x_5,x_6,x_7)=e^{-2i\pi(f_1x_1+f_2x_2+f_3x_3+f_4x_4)}.
	\end{eqnarray*}
	For $F\in L^1(N_{7}),$ let 
	\begin{eqnarray*}
		&&\widehat F(f_1,f_2,f_3,f_4):= \chi_{(f_1,f_2,f_3,f_4)}\\
		&=&\int_{N_{7}}F(x_1,x_2,x_3,x_4,x_5,x_6,x_7)e^{-2i\pi(f_1x_1+f_2x_2+f_3x_3+f_4x_3)}dX
	\end{eqnarray*}
	and 
	$$||F||_{\infty,0}:=\underset{(f_1,f_2,f_3,f_4)\in\R^4}{\sup}\left|\chi_{(f_1,f_2,f_3,f_4)}(F)\right|=||\widehat F||_{\infty}.$$	
\end{itemize}
\begin{defn}
	Let 
	$$\pi_0=\text{ind}_{P}^{N_{7}}\chi_0$$ be the left regular representation of $N_{7}$ on the Hilbert space $L^2(N_{7}/P)\simeq L^2(\R^2).$ Then the image $\pi_0(C^*(N_{7}))$ is just the $C^*$-algebra of $\R^2$ considered as an algebra of convolution operators on $L^2(\R^2)$ and $\pi_0(C^*(N_{7}))$ is isomorphic to the algebra $C_0(\R^4)$ of continuous functions vanishing at infinity on $\R^4$ via the abelian Fourier transform
	$$\widehat F(f_1,f_2,f_3,f_4)=\int_{N_{7,2}}F(x_1,x_2,x_3,x_4)e^{-2i\pi(f_1x_1+f_2x_2+f_3x_3+f_4x_3)}dX.$$
\end{defn}
\begin{defn}
	Let  $a\in C^* (N_{7})$, we define the Fourier transform of $a$, $\widehat a = \F(a)$, as the field of bounded linear operators over the spectrum of $N_{7}$, but also if  combine all the  the unitary characters of $N_{7}$ into the representation
	$\pi_0$. This yields the set $\widehat{N_{7}} :=\Gamma_2\cup\Gamma_1\cup\Gamma_0$ and we define for $a\in C^*(N_{7})$
	the operator field:
	\begin{eqnarray*}
		\widehat a(f_3,f_4,f_7)&=&\F(a)(f_3,f_4,f_7):=\pi_{(f_3,f_4,f_7)}(a)\in\K(L^2(\R^2)),\\ &&(f_3,f_4,f_7)\in\Gamma_2.\\
		\widehat a(f_5,f_6)&=&\F(a)(f_5,f_6):=\pi_{(f_5,f_6)}(a)\in\K(L^2(\R^2)),\ \ (f_5,f_6)\in\Gamma_1.\\
		\widehat a(0)&=&\F(a)(0):=\pi_0(a)\in C^*(\R^4)\subset\B(L^2(\R^4)).\\
	\end{eqnarray*}
	Let $ l^{\infty}(\widehat{N_{7}}) $ be the $ C^* $-algebra of all uniformly bounded  fields of 
	bounded linear operators defined over $ \widehat{N_{7}}=\Gamma_2\cup\Gamma_1\cup\Gamma_0 $, i.e.
	\begin{eqnarray}
		\nn l^{\infty}(\widehat{N_{7}})&:=&\{(\ph(\ga)\in \B(\H_\ga))_{\ga\in\widehat{N_{7}}}, 
		||\ph||_{\infty}:=\sup_{\ga\in\widehat{N_{7}}}\noop{\ph(\ga)}\}. 
	\end{eqnarray}
\end{defn}
\begin{rem}
	The Fourier transform $  \F:C^*(N_{7})\to l^{\infty}(\widehat{N_{7}})$ is an isometric
	homomorphism 
	of $ C^* $-algebras. Our aim is now to determine the image of the mapping $ \F $
	inside the algebra $ l^{\infty}(\widehat{N_{7}}) $. For that, we are establishing 
	certain properties of the operator fields $ \F(F),F\in \S_c $. There are three different types. We shall call them:  continuity, infinity, and changing of layers conditions. These properties
	will
	be shown later to be sufficient to determine the image of $ \F $.
\end{rem}
\subsection{The changing of layers condition.}$ $\\ We shall study   now the behaviour of the operators $\pi_{f_k}(F)$ for a properly  converging sequence $\O_{f_k}\subset\GA_i$ with limit points in $\GA_{i-1}$, $i=1,2$.
\subsubsection{\bf{The generic case, $c_1\neq 0$.}}
We first consider the case of a properly converging sequence $\ol\O=
(\O_{f_k})_k\subset  \Gamma_{2}
$ with limit points outside $ \Gamma_{2} $. 
We knew from Theorem \ref{2.2}, that  $L(\ol\O)=\{\O_{(f_5,-\frac{c_1}{f_5})}, 
f_5\in\R^*\} .$
\begin{defn} $  $
	\rm   
	\begin{itemize}
		
		\item 
		Let $ p_k:=(f_k)_{|\p_{(f_3,f_4,f_7)}}, k\in\N, $ and let $L=L(\ol\O)_{|\p_{(f_3,f_4,f_7)}}$.
		\item   Let for $ (x_1,x_2)\in\R^{2},p_{(f_3,f_4,f_7)}\in\p_{(f_3,f_4,f_7)}^*, $
		\begin{eqnarray} 
			\nn (x_1,x_2)\cdot p&:=&\text{Ad}^*((x_1,x_2))p_{(f_3,f_4,f_7)}. 
		\end{eqnarray}
		
	\end{itemize}
	
\end{defn}
\begin{defn}$  $
	\begin{enumerate}
		\item   For $ k\in\N, j\in \Z^*, $ let
		\begin{eqnarray*}
			i_k&:=&
			\frac{c_1}{f_7^kf_4^k} ,\\
			\nu_k&=&
			\nonumber \left\{\begin{array}{cc}
				|{1-{i_k}}| &\text{ if }1-{i_k}\ne 0\\
				\frac{1}{k+1} &\text{ if }1-{i_k}= 0\\
			\end{array}
			\right.
			\\
			\ve_k&:=&
			\left\{\begin{array}{cc}
				\frac{|f_{7}^k|}{\nu_k^{1/4}} &\text{ if }{|f_{7}^k|}< \nu_k\\
				\vert f_{7}^k \vert R_k&\text{ if }{|f_{7}^k|}\geq \nu_k.\\
			\end{array}
			\right.\\
			\ f_{5,j}^k&:=&
			j\ve_k,\\
			\de_k&:=&
			\nonumber \left\{
			\begin{array}{cc}
				\frac{|f_{7}^k|}{(\nu_k)^{\frac{1}{4}}}&\text{ if }{|f_{7}^k|}< \nu_k\\
				\nu_k^{\frac{3}{4}} &\text{ if }{|f_{7}^k|}\geq \nu_k.\\
			\end{array}
			\right.
		\end{eqnarray*}
		Here $ (R_k)_k $ is a real sequence, such that $ \lim_k R_k=+\infty $, $ \lim_k R_k^2 |f_{7}^k|=0,\lim_k R_k \nu_k^{\frac{3}{4}}=0 $. 
		\item Let
		\begin{eqnarray*}
			x^{k}_{1,j}&:=&
			\frac{f_{5,j}^k}{f_{7}^ki_k},\\
			x^{k}_{2,j}&:=&-
			\frac{c_1}{f_{5,j}^kf_{7}^k}, \\
			g^{k}_j&:=&
			(x^k_{1,j},x^{k}_{2,j}).
		\end{eqnarray*}
		Then:
		\begin{eqnarray*}
			g^k_j\cdot p_k&=&(f_3^k+\frac{1}{2}f_7^k(x_{1,j}^{k})^2-(x_{2,j}^k)^2)X_3^*+\frac{f_{5,j}^k}{i_k}X_5^*-\frac{c_1}{f_{5,j}^k}X_6^*+f_{7}^kX_{7}^*\\
			&=&(f_3^k+\frac{(f_{5,j}^k)^2}{2c_1i_kf_4^k}-\frac{c_1i_kf_4^k}{2(f_{5,j}^k)^2})X_3^*+\frac{f_{5,j}^k}{i_k}X_5^*-\frac{c_1}{f_{5,j}^k}X_6^*+f_{7}^kX_{7}^*,\\ &&  k\in\N,\ j\in\Z.
		\end{eqnarray*}
		
		\item For $ j\ne 0, k,m\in \N^*$ let: 
		\begin{eqnarray*}
			& &U_{j,m}^k\\&:=&
			\left\{(x_1,x_2)\in\R^2;|f_{7}^kx_1+j\frac{\varepsilon_k}{|i_k|}|\leq\frac{\ve_k}{|i_k|}\ \text{and}\ \vert
			f_{7}^k x_2+\frac{c_1}{f^k_{5,j}}\vert\leq m\de_k   \right\}.
		\end{eqnarray*}
		
		\item 
		Let   
		\begin{eqnarray*}
			U^{k}_m&:=&\bigcup_{j\in \Z^*}
			U^{k}_{j,m}\\
		\end{eqnarray*}
		
		\item  Let
		\begin{eqnarray*}
			p^k_j
			&:=&(f_3^k+\frac{1}{2}
			(\frac{(f_{5,j}^k)^2}{f_7^k i_k^2}-\frac{c_1^2}{f_7^k (f_{5,j}^k)^2}))X_3^*+\frac{f_{5,j}^k}{i_k}X_5^*-\frac{c_1}{f_{5,j}^k}X_6^*\\
			&=&(f_3^k+\frac{1}{2}(f_7^k(x_{1,j}^k)^2-(x_{2,j}^k)^2))X_3^*+f_7^kx_{1,j}^kX_5^*+f_7^kx_{2,j}^kX_6^*\in\p^*,\\ &&
			j\in\Z^*,k\in\N.
		\end{eqnarray*}
		\end{enumerate}
\end{defn}
	
\begin{prop}\label{prop1}
	$ $  
	\begin{enumerate}
		\item[1.]  Let take the compact sub-set $ K\subset N_{7}$.   
		For $ k $ large enough we get
		\[KU^{k}_{j,m}\subset U^{k}_{j,m+1}\cup U^{k}_{j-1,m+1}\cup
		U^{k}_{j+1,m+1}.
		\] 
		\item[2.]  We take the compact sub-set $ C\subset \p_{(f_4,f_7)}^*$. Then there are $ \ta>0 $  such that for $ k $ large enough and for every $ (x_3,x_4,x_5,x_6,f_7^k)\in C\cap N_{7}\cdot p_k  $,  we have $ \{|x_5|,|x_6| \}\subset [\ta, \frac{1}{\ta}] . $
	\end{enumerate}
	
\end{prop}
\begin{proof} $ $
	\begin{enumerate} 
		\item[1.] We can assume that $ KP \subset [-M,M] ^{5}P $ for certain $
		M>0 $. Then,
		\begin{eqnarray*}\label{7ama}
			\begin{cases}
				\frac{j \ve_k}{\vert i_k \vert}-\frac{\ve_k}{ |i_k|}\leq -x_1{f_{7}^k}<\frac{j \ve_k}{\vert i_k \vert}+\frac{\ve_k}{|i_k|},\\
				\vert f_{7}^kx_2+\frac{c_1}{ f^k_{5,j}}\vert\leq{m\de_k}. 
			\end{cases}
		\end{eqnarray*}
		Therefore, since $ |v|<M $,  we have for $ k $ large enough that
		\begin{eqnarray*}
			&&
			-x_1f_{7}^k- vf_{7}^k<\frac{j\ve_k}{\vert i_k \vert}+\frac{\ve_k}{|i_k|}+\vert f_{7}^k\vert M <\frac{(j+1)\ve_k}{\vert i_k \vert}+\frac{\ve_k}{|i_k|},\\
			&&
			\frac{(j-1)\ve_k}{\vert i_k\vert}-\frac{\ve_k}{|i_k|}\leq -x_1f_{7}^k-vf_{7}^k,\\
			&&
			\vert f_{7}^kx_1+f_{7}^k u+\frac{c_1}{ f^k_{5,j}}\vert<m\de_{k} +M|f_{7}^k|
			<(m+1)\de_k.
		\end{eqnarray*}
		Hence, for $ k $ large enough: 
		\[KU^{k}_{j,m}\subset U^{k}_{j-1,m+1}\cup U^{k}_{j,m+1}\cup
		U^{k}_{j+1,m+1}.
		\]  
		\item[2.]  Take $ M>0 $ such that $ C\subset [-M,M] ^5 $. Since for $ (x_3,x_4,x_5, x_6 ,f_{7}^k)=(x_1,x_2)\cdot p_k \in N_{7}\cdot p_k\cap C$ we have that:
		\begin{eqnarray*}
			&&|x_4|=|f_4^k+f_{7}^kx_1x_2 |\leq M, |f_{7}^k x_1|\leq M, |f_{7}^k x_2|\leq M\\
			&\Rightarrow&|f_7^kf_{4}^k+f_{7}^kx_1f_{7}^kx_2|\leq M|f_{7}^k|\\
			&\Rightarrow&M|x_6|\geq |x_5||x_6|\geq \frac{| c_{1}|}{2} (\textrm{ for }k\textrm{ large enough})\\
			&\Rightarrow&|x_6|\geq \frac{ |c_{1}|}{2 M}.
		\end{eqnarray*}
		
	\end{enumerate}	
\end{proof}	
\begin{Nt}\label{not1}
	$  $
	\begin{itemize}
		\item For $ k\in\N, j\in\Z^* $ let
		\begin{eqnarray*}
			V^{k}_{j,m+1}&:=&U^{k}_{j-1,m+1}\cup U^{k}_{j,m+1}\cup
			U^{k}_{j+1,m+1}.
		\end{eqnarray*}
		\item  
		For all $m\in\Z^*$ we denote by $ R^{k}_m $ the box 
		$$R^k_{m}:=\left[-\frac{\ve_k}{|i_kf_{7}^k|},\frac{\ve_k}{|i_kf_{7}^k|} \right]\times\left[-\frac{m {\de_k}
		}{|f_{7}^k|},\frac{m {\de_k}
		}{|f_{7}^k|}\right]. $$
	\end{itemize}
	\end{Nt}

\begin{lem} \label{lem1} $ $
	Let  $ \ta>0 $. Then we have for $ k,m\in\N ^*$ large enough,  for any  $ j\in \Z^*$
	such that $ 
	\frac{1}{\ta}>\vert j\vert \ve_k>\ta$, that $
	U_{j,m}^{k}\subset
	R^{k}_m+ g^{k}_j $.
\end{lem}
\begin{proof}
	Let $(x_1,x_2)\in U_{j,m}^k$. Then for $ k $ large enough:
	\begin{eqnarray*}
		&& \begin{cases}
			\frac{f_{5,j}^k}{\vert i_k \vert}-\frac{\ve_k}{|i_k|}\leq -f_{7}^k x_1<\frac{f_{5,j}^k}{|i_k|}+\frac{\ve_k}{ |i_k|}\\
			\vert f_{7}^kx_2+\frac{c_1}{f^k_{5,j}}\vert<m\de_k 
		\end{cases}\\
&&	\begin{cases}
	\Rightarrow\vert
	x_1-x_{1,j}^k\vert\leq \frac{\ve_k}{|i_{k}f_{7}^k|}\Rightarrow x_1\in [\frac{\ve_k}{|i_kf_{7}^k|},\frac{\ve_k}{|i_kf_{7}^k|}]+x_{1,j}^k\\	
\Rightarrow\vert
x_2-x_{2,j}^k\vert<\frac{m\de_k}{|f_{7}^k|}	\Rightarrow
	x_2\in [-\frac{m\de_k}{|f_{7}^k|},\frac{m\de_k}{|f_{7}^k|}]+x_{2,j}^k    
	\end{cases}
\end{eqnarray*}
$\Longrightarrow s\in R^k_{m}+g_j^k.$
\end{proof}
\begin{lem} \label{lem2} $  $
	Let $ \ta>0 $, for  $ k\in\N $ large enough, for $ j\in \Z^*$ such that $\frac{1}{\ta}>f^k_{5,j}|
	>\ta$ and, any $ (x_1,x_2)\in U^k_{j,1} $ we obtain
	\begin{eqnarray} 
		\nn \Vert (x_1,x_2)\cdot p_k-((x_1,x_2)\cdot ({g^k_j})^{-1})\cdot p^{k}_{j}\Vert&\leq&
		\gamma_k,
	\end{eqnarray}
	where, $\gamma_k=\max(\al_k,\be_k)$,\\ $\al_k=\frac{\nu_k^{1/2}}{2|i_k|^2}+\frac{\nu_k^{1/2}}{2}
	+\frac{\nu_k^{1/2}}{|i_k|}+\frac{\ve_k}{|i_k|}+\de_{k}+|f_7^k|\underset{n\to\infty}{\longrightarrow}0$ 	and $\be_k=\frac{R^2_k|f_7^k|}{2|i_k|^2}+\frac{\nu_k^{1/2}
	}{2}+\frac{R_{k}\nu_k^{\frac{1}{2}}}{|i_k|}+\frac{\ve_k}{|i_k|}+\de_{k}+|f_{7}^k|\underset{n\to\infty}{\longrightarrow}0.$
\end{lem}
\begin{proof}
	Indeed we have that $ (x_1,x_2)=(x^k_{1,j}+x_1',x^k_{2,j}+x_2') $ such that $ |x_1'|\leq \frac{\ve_k}{|i_kf_{7}^k|} $ and $ |x_2'|\leq  {\frac{\de_k}{|f_{7}^k|}}$. Therefore,
	\begin{eqnarray} 
		\nn &&(x_1,x_2)\cdot p_k-(x_1',x_2')\cdot p^{k}_{j}\\
		\nn &=&
		\Big(f_3^k+f_7^k\frac{(x_1'+x_{1,j}^k)^2}{2}-f_{7}^k\frac{(x_2'+x_{2,j}^k)^2}{2},f_4^k+f_{7}^k(x_1'+x_{1,j}^k)(x_2'+x_{2,j}^k),\\  &\nn & 
		f_7^k(x_1'+x_{1,j}^k),f_7^k(x_2'+x_{2,j}^k),f_7^k\Big)
		-
		\Big(f_3^k+f_7^k(\frac{(x_{1,j}^k)^2-(x_{2,j}^k)^2}{2})+f_{7}^kx_{1,j}^kx'_1\\ \nn && -f_{7}^kx_{2,j}^kx_2',f_{7}^kx_{1,j}^kx_2'+f_7^kx_{2,j}^kx_1',f_{7}^kx_{1,j}^k,f_{7}^kx_{2,j}^k,0\Big)\\ 
		\nn &=& 
		\bigg(f_3^k+f_7^k\frac{(x_1')^2}{2}+f_{7}^k\frac{(x_{1,j}^k)^2}{2}+f_7^kx_1'x_{1,j}^k-f_7^k\frac{(x_2')^2}{2}-f_7^k\frac{(x_{2,j}^k)^2}{2}-f_7^kx_2'x_{2,j}^k, \\ & &\nn f_4^k +f_{7}^kx_1'x_2'+f_7^kx_{1,j}^kx_2'+ f_7^kx_1'x_{2,j}^k,f_7^kx_1'+f_7^kx_{1,j}^k,f_7^kx_2'+f_7^kx_{2,j}^k,f_7^k\bigg)\\ \nn
		&&-\bigg(f_3^k+f_7^k(\frac{(x_{1,j}^k)^2-(x_{2,j}^k)^2}{2})+f_{7}^kx_{1,j}^kx'_1-f_{7}^kx_{2,j}^kx_2',f_{7}^kx_{1,j}^kx_2'+f_7^kx_{2,j}^kx_1',\\ \nn & &f_{7}^kx_{1,j}^k,f_{7}^kx_{2,j}^k,0\bigg)\\
		\nn&=&\left(f_7^k\frac{x_1^{'2}}{2}-f_7^k\frac{x_2^{'2}}{2},f_{7}^kx_1'x_2',f_7^kx_1',f_7^kx_2',f_7^k\right)
	\end{eqnarray}
	and so, for $ k $ large enough,  if $ |f_{7}^k| <\nu_k $:
	\begin{eqnarray*} 
		&& 
		\Vert (x_1,x_2)\cdot p_k-((x_1,x_2)\cdot ({g^k_j})^{-1})\cdot p^{k}_{j}\Vert\\
		&\leq&
		|f_7^k\frac{x_1^{'2}}{2}|+|f_7^k\frac{x_2^{'2}}{2}|+|f_{7}^kx_1'x_2'|+|f_7^kx_1'|+|f_7^kx_2'|+|f_7^k|\\	&\leq&
		\frac{\ve_k^2}{2|i_k|^2|f_7^k|}+\frac{\de_k^2}{2|f_7^k|}+\frac{\ve_k\de_{k}}{|i_k||f_7^k|}+\frac{\ve_k}{|i_k|}+\de_{k}+|f_{7}^k|\\
		&\leq&
		\frac{|f_7^k|}{2|i_k|^2\nu_k^{1/2}}+\frac{|f_7^k|}{2\nu_k^{1/2}}+\frac{|f_7^k|}{|i_k|\nu_k^{1/2}}+\frac{\ve_k}{|i_k|}+\de_{k}+|f_{7}^k|\\
		\nn &\leq& \frac{\nu_k^{1/2}}{2|i_k|^2}+\frac{\nu_k^{1/2}}{2}
		+\frac{\nu_k^{1/2}}{|i_k|}+\frac{\ve_k}{|i_k|}+\de_{k}+|f_7^k|\\ \nn	&\leq&
		\al_k.
	\end{eqnarray*}
	If $|f_{7}^k| \geq\nu_k $, then
	\begin{eqnarray*} 
		&& 
		\Vert (x_1,x_2)\cdot p_k-((x_1,x_2)\cdot ({g^k_j})^{-1})\cdot p^{k}_{j}\Vert\\
		&\leq&
		|f_7^k\frac{x_1^{'2}}{2}|+|f_7^k\frac{x_2^{'2}}{2}|+|f_{7}^kx_1'x_2'|+|f_7^kx_1'|+|f_7^kx_2'|+|f_7^k|\\
		&\leq&\frac{R^2_k|f_7^k|}{2|i_k|^2}+\frac{\nu_k^{1/2}
		}{2}+\frac{R_{k}\nu_k^{\frac{1}{2}}}{|i_k|}+\frac{\ve_k}{|i_k|}+\de_{k}+|f_{7}^k|\\
		&\leq&\be_k
	\end{eqnarray*}
	
\end{proof}
\begin{defn}
	For a measurable subset $S\subset\R^{n}$, let $M_S$ be the
	multiplication operator on $ L^{2}(\R^{n}) $ with the characteristic function
	of the set $S$ $(n\in\N^*)$.\\
\end{defn}
\begin{defn}
	The sub-algebra $\p_{(f_5,f_6)}$ is a polarization at $\ell_{f_5,f_6}$, which gives us representation 
	\begin{center}
		$\pi_{\ell_{(f_5,f_6)}}:=ind^{N_{7}}_{P_{\ell_{f_5,f_6}}}\chi_{\ell_{(f_5,f_6)}}\in \widehat{N_{7}}$
	\end{center} and we define the representation$\tilde{\sigma}_{\ell}$ by: \begin{center}
		$\tilde{\sigma}_{\ell_{(f_3,f_4,f_7)}}:=ind^{N_{7}}_{P_{\ell_{f_3,f_4,f_7}}}\chi_{\ell_{f_3,f_4,f_7}}$.
	\end{center}
	Hence for every $a\in C^*(N_{7})$ we have that: 
	\begin{eqnarray*}\label{noes}
		\noop{\tilde\si_{\ell_{(f_3,f_4,f_7)}}(a)}&=&\noop{\pi_{\ell_{(f_3,f_4,f_7)}}(a)}
	\end{eqnarray*}
	and
	\begin{eqnarray*}
		\si(\widehat{a}_{|L(\ol O)})=\tilde \si(a).
	\end{eqnarray*}
\end{defn}
\begin{defn}\label{piispip}$  $\rm
	\begin{itemize}\label{defsiptw}

		\item   We denoted by  $C_{\ol \O}= CB(L(\ol\O),\B(L^{2}(\R^{2}))) $ the $ C^* $-algebra of all
		continuous, uniformly bounded mappings $ \ph:L(\ol\O)\mapsto \B(L^{2}(\R^2)) $ from
		the locally compact space $L(\ol\O)$ into the algebra of bounded linear
		operators $ \B(L^2(\R^2)) $ on the Hilbert space $ L^2(\R^2) $.
		We find that for every $ a\in C^*(N_{7}) $, the operator field $ \hat{a}_{|L(\ol\O)} \subset C_{\ol\O} $. Also, for our  $ p=f_5X_5^*+f_6X_6^*\in \p_{(f_3,f_4,f_7)}^*$, we get a representation $ \si_\ell $ on the Hilbert space $ L^2(\R^2) $ of the algebra $ C_{\ol \O} $ given by the following formula:
		\begin{eqnarray}\label{sielldef}
			\si_{\ell_{f_4,f_8,f_9}}(\ph)\xi:=ind^{N_{7}}_{P_{(f_5,f_6)}}\chi_{\ell_{(f_5,f_6)}}(\ph)\xi, \forall \ph\in C_{\ol \O}.
		\end{eqnarray}
		Then $ \noop{\si_\ell(\ph)}\leq  ||\ph||, \ph\in C_{\ol \O} $.

		\item   Define for $ k\in \N, \ph\in C_{\ol\O} $, the linear
		operator  $ \si_{k,\ol\O}(\ph)=\si_{k}(\ph) $ by
		\begin{eqnarray*} 
			\nn  \si_{k,\ol\O}(\ph):=\sum_{j\in \Z^*}M_{V^{k}_{j,m+1}}\circ
			\si_{{(g^{k}_{j}}^{-1})\cdot p^{k}_j}(\ph)\circ M_{U^{k}_{j,m}},
		\end{eqnarray*}
		where $ \si_\ell, \ell\in\n_{7}^*, $ is an in Equation (\ref{sielldef}).
	\end{itemize}
	
\end{defn}
\begin{prop}\label{noolO}
	We have for any $ \ph\in C_{\ol \O} $ that
	\begin{eqnarray*}
		\noop{\si_{k,\ol\O}(\ph)}&\leq&\sqrt 3 \sup_{\ell\in L(\ol
			O)}\noop{\ph(\ell)}.
	\end{eqnarray*}
\end{prop}
\begin{proof}
	Since $ V^{k}_{j,m+1}$ is the disjoint union of the measurable
	sets  $U^{k}_{j+i,m+1 }, i=-1,0,1$, it follows that for any $ \xi\in
	L^{2}(\R^{2}) $:
	\begin{eqnarray*}
		\Vert \si_{k,\ol\O}(a)(\xi)\Vert^{2}&=&\sum_{i=-1}^{+1}\sum_{j\in
			\Z^*}\Vert
		M_{U^{k}_{j+i,m+1}}( \si_{(g^{k}_{j})^{-1}\cdot p^{k}_j}(\ph)(
		M_{U^{k}_{j,m}}(\xi))\Vert^{2}\\
		\nn &\leq&\sum_{i=-1}^{+1}\sum_{j\in \Z^*}\noop{ \si_{(g^{k}_{j})^{-1}\cdot
				p^{k}_j}(\ph)}^{2}\Vert
		M_{U^{k}_{j,m}}(\xi)\Vert^{2}\\
		\nn &\leq&3 \sup_{\ell\in L(\ol \O)}\noop{\ph(\ell)}^2\\
		\nn &&(\textrm{by Definition \ref{piispip}}).
	\end{eqnarray*}	
\end{proof}
\begin{prop}\label{apprve}
	Let take $a\in C^*(N_{7})$ then,
	\begin{eqnarray*}\label{twiti}
		\lim_{k\to\infty}\noop{\pi_{\ell_k}(a)-\tilde\si_{k,\ol\O}(a)}=0.
	\end{eqnarray*}
\end{prop}
\begin{proof} We choose $ \ve>0$. Then, we consider $ F\in \S_c $.
	Let take $ C$ a compact subset of $\subset \p_{(f_3,f_4,f_7)}^* $ and an $ M>0 $
	where $ [-M,M] ^{2}\times C $ is the support of the function $\R^{2}\times \p_{(f_3,f_4,f_7)}^*\ni (x_1,x_2,p)\to \hat F^{P_{(f_3,f_4,f_7)}}(x_1,x_2,p) $. From Proposition
	\ref{prop1}  we obtain a $ \ta>0  $, such that for $ k $ large enough:
	\begin{eqnarray} 
		\nn \hat F^{P_{(f_3,f_4,f_7)}}(st^{-1}, t\cdot p_k)=0
		\textrm{ for } t\not\in \bigcup_{|j|>\frac{\ta}{|f_{7}^{k}|}}
		U^{k}_{j,m} \textrm{ or }st^{-1} \not \in ([-M,M] ^{2})P.
	\end{eqnarray}
	Then, for $ k $ large enough :
	\begin{eqnarray} 
		\nn \pi_k(F)&=&\pi_k(F )
		\circ M_{U^{k}_{m}}
	\end{eqnarray}
	and once again by Proposition \ref{prop1} and Notations \ref{not1}
	\begin{eqnarray*} 
		\nn M_{V^{k}_{j,m+1}}\circ \pi_k(F)\circ M_{U_{j,m}^k}&=&\pi_k(F)\circ
		M_{U_{j,m}^k}, k\in\N, j\in \Z^*.
	\end{eqnarray*}
	Hence, for $ \xi\in L^{2}(N_{7}/P_{f_3,f_4,f_7},\ch_{\ell_k}), s\in N_{7} $:
	\begin{eqnarray*}
		& &\left(\pi_k(F)\circ M_{U_{j,m}^k}-M_{V^{k}_{j,m+1}}\circ
		\si_{(g^{k}_{j})^{-1}\cdot p^{k}_j}(F)\circ M_{U^{k}_{j,m}}\right)\xi(s)\\
		\nn &=&1_{V^{k}_{j,m+1}}(s)\int_{U^{k}_{j,m}}\Big(\hat
		F^{P_{f_3,f_4,f_7}}(st^{-1},t\cdot p_k)\\ &&-\hat F^{P_{f_3,f_4,f_7}}(st^{-1},t(g^{k}_{j})^{-1}\cdot
		p^k_{j})\Big) \xi(t)d\dot t.
	\end{eqnarray*}
	Since the function $ (s,p)\to\vert \hat F^{P_{f_3,f_4,f_7}}(s,p)\vert^2 $ is contained in $
	C^{\infty}_c(N_5/P,\p_{f_3,f_4,f_7}^*) $ there are  a non-negative continuous functions with
	compact support  \\  $ \va:N_{7}/P_{f_3,f_4,f_7} \to \R_+$ where for every $ q,p\in\p_{f_3,f_4,f_7}^* , s\in N_{7}$:
	\begin{eqnarray*}
		\vert \hat F^{P_{f_3,f_4,f_7}}(s, q)-\hat F^{P_{f_3,f_4,f_7}}(s, p)\vert&\leq&\va(s)|| q- p||.
	\end{eqnarray*}
	It then follows from Lemma \ref{lem1} and Lemma \ref{lem2}
	that for  $ k\in \N $ large enough,\\ $  j\in \Z^*, \frac{1}{\ta}>|jf_{7}^{k}|>\ta, s\in N_{7}/P_{f_3,f_4,f_7} $:
	\begin{eqnarray*}
		&&1_{V^{k}_{j,m+1}}(s)\int_{U^{k}_{j,m}}\left\vert\hat F^{P_{f_3,f_4,f_7}}(st^{-1},t\cdot
		p_k)-\hat F^{P_{f_3,f_4,f_7}}(st^{-1},t(g^{k}_{j})^{-1}\cdot p^k_{j})\right\vert d\dot
		t\\
		\nn &\leq&\gamma_k\int_{N_{7}/P_{f_3,f_4,f_7}}\va(st^{-1})d\dot t
	\end{eqnarray*}
	such that, $\gamma_k=\max(\al_k,\be_k)$,\\ $\al_k=\frac{\nu_k^{1/2}}{2|i_k|^2}+\frac{\nu_k^{1/2}}{2}
	+\frac{\nu_k^{1/2}}{|i_k|}+\frac{\ve_k}{|i_k|}+\de_{k}+|f_7^k|\underset{n\to\infty}{\longrightarrow}0$ 	and $\be_k=\frac{R^2_k|f_7^k|}{2|i_k|^2}+\frac{\nu_k^{1/2}
	}{2}+\frac{R_{k}\nu_k^{\frac{1}{2}}}{|i_k|}+\frac{\ve_k}{|i_k|}+\de_{k}+|f_{7}^k|\underset{n\to\infty}{\longrightarrow}0.$
	For every $ t\in N_{7,2}/P_{f_3,f_4,f_7} $:
	\begin{eqnarray*}
		&&1_{U^{k}_{j,m}}(t)\int_{V^{k}_{j,m+1}}\left\vert\hat F^{P_{f_3,f_4,f_7}}(st^{-1},t\cdot
		p_k)-\hat F^{P_{f_3,f_4,f_7}}(st^{-},t(g^{k}_{j})^{-1}\cdot p^k_{j})\right\vert d\dot
		s\\
		\nn &\leq&\gamma_k\int_{N_{7}/P_{f_3,f_4,f_7}}\va(st^{-1})d\dot s\\
		\nn &=&\gamma_k||\va||_1.
	\end{eqnarray*}
	Thanks  to Young's estimate, we observe   that (for every $ k\in\R, j\in \Z^*$):
	\begin{eqnarray*}
		\noop{\pi_k(F)\circ M_{U^{k}_{j,m}}^{k}-(M_{V^{k}_{j,m+1}}\circ
			\si_{(g^{k}_{j})^{-1}\cdot p^{k}_j}(F)\circ
			M_{U^{k}_{j,m}})}&\leq&\gamma_k||\va||_1.
	\end{eqnarray*}
	
	Therefore, for $ k $  large enough:
	\begin{eqnarray*}\label{twiti1}
		&&\noop{\pi_{\ell_k}(F)-\tilde\si_{k,\ol\O}(F)}\\
		\nn &=&\noop{\sum_{j\in \Z,|jf_{10}^{k}|>\ta} \left(M_{V^{k}_{j,m}}\circ
			\pi_{\ell_k}
			(F)\circ M_{U^{k}_{j,m}}-M_{V^{k}_{j,m}}\circ
			\si_{(g^{k}_{j})^{-1}\cdot p^{k}_j}(F)\circ M_{U^{k}_{j,m}}\right)}\\
		\nn &\leq&(\sum_{i=-1}^{{+1}}\sup_{j\in \Z,|jf_{7}^{k}|>\ta
		}\noop{M_{U^{k}_{j+i,m+1}}\circ \pi_{\ell_k} (F)\circ
			M_{U^{k}_{j,m}}-M_{U^{k}_{j+i,m+1}}\circ \si_{(g^{k}_{j})^{-}\cdot
				p^{k}_j}(F) \\ \nn &&\circ M_{U^{k}_{j,m}}}^{2})^{1/2}\\
		\nn &\leq&\sqrt 3\gamma_k||\va||_1.
	\end{eqnarray*}
	Let take now $ a\in C^*(N_{7}) $ and $ \ve>0 $. Let  $ F\in\S_c $, such that
	$ ||F-a||_{C^*(N_{7})}<\ve$.
	Then, there is an index $ k_\ve\in\N $ such that $ \noop{\tilde\si_{k,
			\ol\O}(a)-\tilde\si_{k,\ol\O}
		(F)}<\ve $ for $ k\geq k_\ve $. Hence for $ k\geq k_\ve $:
	\begin{eqnarray*}
		\noop{\tilde\si_{k,\ol\O}(a)-\pi_k(a)}&\leq&\noop{\tilde\si_{k,
				\ol\O}(a)-\tilde\si_{k,\ol\O}
			(F)}+\noop{\pi_k(a-F)}\\ \nn && +\noop{\tilde\si_{k,\ol\O}(F)-\pi_k(F)}\\
		&\leq&3\ve.
	\end{eqnarray*}

\end{proof}
\begin{cor}\label{almosthom}
	Let $ a,b\in  C^*(N_{7})  $. Then,
	\begin{eqnarray*}
		\lim_{k\to\infty}\noop{\tilde\si_{k,\ol\O}(ab)-\tilde\si_{k,\ol\O}
			(a)\circ
			\tilde\si_{k,\ol\O}(b)}=0.
	\end{eqnarray*}
	
\end{cor}
\begin{proof}
	Indeed by Proposition \ref{apprve}, we can find for any $ \ve>0 $ an index $ k_\ve $, such that
	\begin{eqnarray*}
		&&\noop{\pi_k(a)-\tilde\si_{k,\ol\O}(a)}<\ve,\noop{\pi_k(b)-\tilde\si_{k,\ol\O}( b)}<\ve,
		\\
		\nn &&\noop{\pi_k(a \ast b)-\tilde\si_{k,\ol\O}(a\ast b)}<\ve, k\geq
		k_{\ve}.
	\end{eqnarray*}
	Hence
	\begin{eqnarray*}
		&&\noop{\tilde\si_{k,\ol\O}(ab)-\tilde\si_{k,\ol\O}
			(a)\circ
			\tilde\si_{k,\ol\O}(b)}\\
		\nn
		&\leq&\noop{\tilde\si_{k,\ol\O}(ab)-\pi_k(ab)}+\noop{\pi_k(ab)-\tilde\si_{k,
				\ol\O }
			(a)\circ
			\tilde\si_{k,\ol\O}(b)}\\
		\nn
		&\leq&\ve+\noop{\pi_k(a)\pi_k(b)-\tilde\si_{k,
				\ol\O }
			(a)\circ \pi_k(b)}+\noop{\tilde\si_{k,
				\ol\O }
			(a)\circ \pi_k(b)\\ &&-\tilde\si_{k,
				\ol\O }
			(a)\circ
			\tilde\si_{k,\ol\O}(b)}\\
		\nn  &\leq&\ve+||b||_{C^*}\noop{\pi_k(a)-\tilde\si_{k,
				\ol\O }
			(a)}+\sqrt 3 ||a||_{C^*}\noop{ \pi_k(b)-
			\tilde\si_{k,\ol\O}(b)}\\
		\nn  &\leq&(2+\sqrt 3  ||a||_{C^*}+ ||b||_{C^*})\ve.
	\end{eqnarray*}
	
\end{proof}
\subsubsection{\textbf{The generic case, $ c_1=0 $.}}$ $\\
Let take a properly converging sequence $\ol\O= (\O_{\ell_k})_k\subset  \Gamma_2
$  such that $\forall k\in\N$ and 
$\ell_k=(f_3^k,f_4^k,0,0,f_{7}^k), $ where $ \lim_kf_{7}^k=0,$ $ \lim_kf_4^k f_{7}^k
=0
$.
Let before $ p_k:=(\ell_k)_{|\p_{f_3,f_4,f_7}} $. Due to, Theorem \ref{2.2} the restriction $ L(\ol\O)_{|\p_{f_3,f_4,f_7}}$ is the closed set
$L=L(\ol\O)_{|\p_{f_3,f_4,f_7}}=L_4\cup L_5\cup
L_6$, such that  $ L_6=\{(0,f_4,0,f_6,0),f_4\in \R, f_6\in \R^*\},
L_5=\{(0,f_4,f_5,0,0),f_4\in \R, f_5\in \R^*\}$, $
L_4=\{(0,f_4,0,0,0),f_4\in \R\} $.

\begin{Nt}\label{tqc}
	\rm   For $ k,m\in\N^*,\N  $ let
	\begin{eqnarray*}
		R_k&>&0,\\
		\ve_k&:=&|f_{7}^k|R_k^{\frac{1}{2}},\\
		\nu_k&:=&\ve_k+|f_4^kf_{7}^k|\\
		\de_k&:=&R_k^2\nu_k,\\
		S^{k}_{6,m}&:=&
		\{(0,x_4,x_5,x_6, f_{7}^k)\in\p_{(f_3,f_4,f_7)}^*; |x_5|<m\de_k^{\frac{1}{2}},
		|x_6|\geq m\de_k^{\frac{1}{2}}\},\\
		\nn S^{k}_{5,m}&:=&
		\{(0,x_4,x_5,x_6, f_{7}^k)\in\p_{(f_3,f_4,f_7)}^*; |x_5|\geq m\de_k^{\frac{1}{2}},
		|x_6|<m\de_k^{\frac{1}{2}}\},\\
		\nn S^{k}_{4,m}&:=&
		\{(0,x_4,x_5,x_6, f_{7}^k)\in\p_{(f_3,f_4,f_7)}^*; |x_5|<m\de_k^{\frac{1}{2}},
		|x_6|<m(\de_k)^{\frac{1}{2}}\}\\
		S^k_m&:=& S^k_{4,m}\cup S^k_{5,m}\cup S^k_{6,m}.
	\end{eqnarray*}
	Now, we take a sequence $ (R_k) $, verify that
	\begin{eqnarray*}
		\lim_k R_k=+\infty,\ \lim_k R_k\de_k=0.
	\end{eqnarray*}
	This is made possible by the fact that $ \lim_k f_{7}^k=\lim_k f_7^kf_{4}^k=0 $.\\
	Let
	for $ k,m\in\N^*, j\in\Z^*, $
	\begin{eqnarray*}
		I^{k}_{j, 6,m}&:=&\{(0,x_4,x_5,x_6, f_{7}^k)\in S^{k}_{6,m};
		j\ve_k-\ve_k\leq x_6< j\ve_k+\ve_k\},\\
		\nn I^{k}_{j, 5,m}&:=&\{(0,x_4,x_5,x_6, f_{7}^k)\in S^{k}_{5,m};
		j\ve_k-\ve_k\leq  x_5< j\ve_k+\ve_k\},\\
		\nn I^{k}_{j, 4,m}&:=&\{(0,x_4,x_5,x_6, f_{7}^k)\in S^{k}_{ 4,m};
		j\de_k^{\frac{1}{4}}-\de_k^{\frac{1}{4}}\leq
		x_4<j\de_k^{\frac{1}{4}}+\de_k^{\frac{1}{4}}\}.
	\end{eqnarray*}
	Finally:
	\begin{eqnarray*}
		U^{k}_{j,i,m}&=&\{t\in N_{7,2}; t\cdot p_k\in I^{k}_{j,i,m}\}, i=4,5, 6\text{ and }j\in
		\Z^*,\\
		\nn U^{k}_{j,m}&=& U^{k}_{j,6,m}\cup U^{k}_{j,5,m}\cup U^{k}_{j,4,m},\\
		\nn U^{k}_m&:=&\bigcup_{j\in\Z^*} U^{k}_{j,m}.
	\end{eqnarray*}

	For $ k\in\N, j\in\Z$ we have 
	\begin{eqnarray*}
		x^{k}_{1,j,6}&:=
		&\frac{j\ve_k}{f_{7}^k}, \ x^{k}_{2,j,
			6}:=-\frac{f_4^k}{j\ve_k}, \ g^{k}_{j,
			6}=x^{k}_{1,j,6}X_1+x^{k}_{2,j,6}X_2,\\
		x^{k}_{1,j, 5}&:=&-\frac{f_4^k}{j\ve_k}, x^{k}_{2,j, 5}:=\frac{j\ve_k}{f_{7}^k}, \ g^{k}_{j,
			5}=x^{k}_{1,j,5}X_1+x^{k}_{2,j,5}X_2
	\end{eqnarray*}
	and
	\begin{eqnarray*}
		x^{k}_{1,j, 4}&:=&
		(\frac{-f_4^k+j\de_k^{\frac{1}{4}} )\ve_k}{f_{7}^k}), \ x^{k}_{1,j,
			4}:=-\frac{\textrm{sign}(f_{7}^k)}{R_k^{\frac{1}{2}}f_{7}^k}, \\
		\nn g^{k}_{j, 4}&:=&x^{k}_{1,j,4}X_1+x^{k}_{2,j,4}X_2.
	\end{eqnarray*}
	Let for $ j\in\Z^*, k\in\N $:
	\begin{eqnarray*}
		p^k_{j,6}&:=&(0,0,-\frac{f_4^kf_{7}^k}{j\ve_k} ,j\ve_k
		,0)\\
		p^k_{j,5}&:=&(0,0, j\ve_k,
		-\frac{f_4^kf_{7}^k}{j\ve_k},0),\\
		\nn  p^k_{j,4}&:=&(0,j\de_k^{\frac{1}{4}},
		0,0,0).
	\end{eqnarray*}
	
\end{Nt}
\begin{prop}\label{suppinU2}
	\begin{enumerate}
		\item[]
		\item[1.]  We consider a compact sub-set $C\subset \p_{(f_3,f_4,f_7)}^* $. Then there are  $ M>0 $  such that for $ k $ large enough, $  m\in\N^*$ and for every $ q=(0,x_{4},x_{5},x_{6},f_{7}^{k})\in C\cap N_{7}\cdot p_k  $,  we obtain $ |x_{4}| \leq M $ and $ q\in S^{k}_m $.
		\item[2.]  We take a compact sub-set $K$ and an $ \ve>0 $. We get for $ k $  large enough
		\[KU^{k}_{j,i,m}\subset
		\bigcup_{j'=-1}^{1}U^{k}_{j+j',i,m+1}=:V^{k}_{j,m}, i=5, 6,
		\]
		\[KU^{k}_{j,4,m}\subset
		\bigcup_{i=-(2m+1)}^{(2m+1)}U^{k}_{j+i,4,m+1}=:V^{k}_{j,4,m}. \]
		
	\end{enumerate}

\end{prop}
\begin{proof}
	\begin{enumerate}
		\item[]
		\item [1.] Suppose that $M>0$ where  $C\subset [-M,M]^{4}$. So, for $q=(0,x_{4},x_5,x_6,f_{7}^{k})=(x_{1},x_{2}).p_{k}\in C\cap N_{7}.p_{k}$ we obtain $|x_{4}|\leq M$, $k\in \mathbb{N}^{*}$. \\ If$|x_{5}|\geq m\de_{k}^{\frac{1}{2}}$, then since
		\begin{center}
			$|f_{7}^{k} f_{4}^{k}+ f_{7}^{k}x_{1}f_{7}^{k}x_{2}|\leq |f_{7}^{k}|M \leq |f_{7}^{k}|R_{k}^{\frac{1}{2}}=\varepsilon_{k}, k$ large enough.
			
		\end{center}
		We get,
		\begin{center}
			$m \de_{k}^{\frac{1}{2}}|x_{6}|\leq |x_{5}x_{6}|=|f_{7}^{k}x_{1}f_{7}^kx_{2}|\leq |f_{7}^{k}f_{4}^{k}|+M|f_{7}^{k}|\leq m^{2}\de_{k} $,
		\end{center}
		thus, $|x_{6}|< m\de_{k}^{\frac{1}{2}}$, for $k$ large enough. Hence $q\in S_{5,m}^{k}$. Similarly for the other cases.
		\item[2.] We assume  that $KP\subset [-M,M]^{4}.P$ for certain $M>0$. For $r=(u,v)\in KP\subset N_{7}/P_{(f_3,f_4,f_7)}$ and $s=(x_{1},x_{2})\in U_{m}^{k}$ we have $r.s=(x_{1}+u,x_{2}+v)$, then
		\begin{eqnarray*}
	 &&(rs).p_{k}\\ &=&\big(f_3^k+f_{7}^{k}\frac{(x_{1}+u)^2}{2}-f_7^k\frac{(x_2+v)^2}{2},f_4^k+f_7^k(x_{1}+u)(x_{2}+v),\\ && f_{7}^{k}(x_{1}+u),f_{7}^k(x_{2}+v),f_{7}^{k}\big).
		\end{eqnarray*}
		For $i=6$:
		\begin{eqnarray*}
			\nn && (x_{1},x_{2}) \in U_{j,6,m}^{k}\\
			&\Leftrightarrow& (x_{1},x_{2}).p_{k}\in I_{j,6,m}^{k}\\
			\nn &\Rightarrow &\begin{cases}
				j \varepsilon_{k} - \varepsilon_{k} \leq  -f_{7}^{k} x_{1} <j \varepsilon_{k} + \varepsilon_{k} \\
				|f_{7}^{k}x_{2}|<m \de_{k}^{\frac{1}{2}}
			\end{cases} \\
			\nn &\Rightarrow& \begin{cases}
				(j-1)\varepsilon_{k} - \varepsilon_{k} \leq f_{7}^{k}x_{1}+f_{7}^{k}u<(j+1)\varepsilon_{k} +\varepsilon_{k}\\
				|f_{7}^{k}x_{2}+f_{7}^{k}v|<(m+1)\de_{k}^{\frac{1}{2}}. \end{cases}
		\end{eqnarray*}
		Then, we observed that $KU_{j,6,m}^{k}\subset \bigcup_{i=-1}^{1} U_{j+i,6,m+1}^{k} $ for $k$ large enough. The same for $i=5$.\\
		For $i=4$:
		\begin{eqnarray*}
			\nn && (x_{1},x_{2}) \in U_{j,4}^{k}\\
			\nn &\Leftrightarrow& (x_{1},x_{2}).p_{k} \in I_{j,6,m}^{k}\\
			\nn &\Rightarrow& \begin{cases}
				|f_{7}^{k}x_{1}|<m\de_{k}^{\frac{1}{2}}, |f_{7}^{k}x_{2}|<m\de_{k}^{\frac{1}{2}}\\
				\de_{k}^{\frac{1}{4}}j-\de_{k}^{\frac{1}{4}}\leq f_{4}^{k}- f_{7}^{k}x_{1}x_{2}< 	\de_{k}^{\frac{1}{4}}j+\de_{k}^{\frac{1}{4}}.
			\end{cases}
		\end{eqnarray*}
		Since for $k$ large enough:
		\begin{center}
			$	|f_{7}^{k}(x_{1}+ u)|<(m+1)\de_{k}^{\frac{1}{2}}, $ $|f_{7}^{k}(x_{2}+v)|<(m+1)\de_{k}^{\frac{1}{2}}$
		\end{center}
		and
		\begin{center}
			$	|f_{7}^{k}ux_{2}|+ |f_{7}^{k}vx_{1}|+ |f_{7}^{k}uv|\leq m \de_{k}^{\frac{1}{2}}|u|+m\de_{k}^{\frac{1}{2}}|v|+|f_{7}^{k}||uv|< \de_{k}^{\frac{1}{4}}$.
		\end{center}
		This means that
		\begin{center}
			$-\de_{k}^{\frac{1}{4}}+	j\de_{k}^{\frac{1}{4}}-\de_{k}^{\frac{1}{4}}\leq f_{4}^{k} - f_{7}^{k}x_{1}x_{2}+f_{7}^{k}ux_{2}+ f_{7}^{k}vx_{1}+ f_{7}^{k}uv< \de_{k}^{\frac{1}{4}}+ j	\de_{k}^{\frac{1}{4}}+\de_{k}^{\frac{1}{4}} .$
		\end{center}
		It implies  that $KU_{j,4,m}^{k}\subset \bigcup_{i=-1}^{1} U_{j+i,4,m+1}^{k}$.
	\end{enumerate}
\end{proof}
\begin{Nt}
	Let for  $j\in\Z^*, m,k\in\N^*$ 
	\begin{itemize}	
		\item
		$R^k_{5,m}=\left[-\frac{m\de_k^{\frac 1 2}}{\vert f_{7}^k\vert},\frac{m\de_k^{\frac 1 2}}{\vert f_{7}^k\vert} \right] \times \left[
		-\frac{\varepsilon_k}{\vert f_{7}^k\vert},\frac{
			\varepsilon_k
		}{
			\vert f_{7}^k\vert } \right]$

		\item $
		R^k_{6,
			m}=\left[
		-\frac{\varepsilon_k}{\vert f_{7}^k\vert},\frac{
			\varepsilon_k
		}{
			\vert f_{7}^k\vert } \right]\times  \left[-\frac{m\de_k^{\frac 1 2}}{\vert f_{7}^k\vert},\frac{m\de_k^{\frac 1 2}}{\vert f_{7}^k\vert} \right].$
		
	\end{itemize}
\end{Nt}
\begin{lem}\label{vinbox1}
	Let $ M>0 $. For $ k\in\N^* $ large enough,   for any  $ j\in \Z^*$ with $
	|j\ve_k|\geq\de_k^{\frac 1 2} $
	the set  $ U_{j,i,m}^{k}, i=6,5,$  is contained in $
	R^{k}_{i,m}+g^{k}_{j,i}$.
\end{lem}
\begin{proof}
	Let $s=(x_{1},x_{2}) \in U_{j,6,m}^{k}$. Then,
	\begin{eqnarray*}
		\nn && (x_{1},x_{2},x_{4},x_{5}).p_{k}\in I_{j,9,m}^{k}\\
		\nn &\Leftrightarrow& \begin{cases}
			(j-1)\varepsilon_{k}\leq f_{7}x_{1}<(j+1)\varepsilon_{k} \Rightarrow |x_{1}-x_{1,j,6}^{k}|\leq \frac{\varepsilon_{k}}{|f_{7}^{k}|}\\
			|f_{7}^{k}x_{2}|<m\de_{k}^{\frac{1}{2}}	\end{cases}\\
		&&\begin{cases}
		 \Rightarrow x_{1}\in \left[
		-\frac{\varepsilon_k}{\vert f_{7}^k\vert},\frac{
			\varepsilon_k
		}{
			\vert f_{7}^k\vert } \right]+x_{1,j,6}^k\\	
		\Rightarrow |x_{2}|<\frac{m\de_{k}^{\frac{1}{2}}}{|f_{7}^{k}|}
		\end{cases}\\
		\nn &\Rightarrow& s\in R_{6,m}^{k}+g_{j,6}^{k}.
	\end{eqnarray*}
	Similarly for $i=5$.
\end{proof}
\begin{lem}
	$  $ Let $ M>0 $.
	\begin{enumerate}
		\item[1.]  For  $ k\in\N^* $ large enough, for $ j\in \Z^*,M\geq  
		|j|\ve_k\geq
		\de_k$ and any $ (x_1,x_2)\in U^k_{j,i}, i=5, 6, $ we have
		that
		\begin{eqnarray} 
			\nn \Vert (x_1,x_2)\cdot p_k-((x_1,x_2)\cdot (g^k_{j,i})^{-1})\cdot
			p^{k}_{j,i}\Vert&\leq&
			4m(R_k\de_k)^{\frac{1}{2}}.
		\end{eqnarray}
		\item[2.]
		For  $ k\in\N^* $ large enough, for $ j\in \Z^*,  |j|\de_k^{\frac{1}{4}}\leq
		M$ and any $ (x_1,x_2)\in U^k_{j,4}$ we have
		that
		\begin{eqnarray} 
			\nn \Vert (x_1,x_2)\cdot p_k-(x_1,x_2)\cdot
			p^{k}_{j,4}\Vert&\leq&
			{\frac{1}{R_k}}+3\de_k.
		\end{eqnarray}
	\end{enumerate}
	
\end{lem}
\begin{proof}
	
	We use the same procedure which we used to show Lemma \ref{lem2}
\end{proof}
\begin{defn}\label{vkdef2}$  $\rm
	
	\rm
	We denoted by   $ C_{\ol \O} $  the $ C^* $-algebra of every uniformly bounded mappings\\
	$ \ph:L(\ol \O)\to \B(L^2(\R^2))  $ which are continuous on the subsets $ L_4,L_5 $ and on $ L_6 $.	We can give the definition of  the representation $ \si_\ell $ of the algebra $ C_{\ol \O} $ on $ L^2(\R^2) $ for $ p\in L_6\cup L_5$, similar to (\ref{piispip}).
	
	For $ p=f_4 X_4^*\in L_4$, we obtain the representation $ \si_\ell=\text{ind}_ {P_{(f_5,f_6)}}^{N_{7}} {\chi_{(f_5 X_5^*+f_6X_6^*)}} $ of $ N_5 $,
	which is an integral over a closed subset of $ L(\ol \O) $ and which we can extend thus to the algebra  $ C_{\ol \O} $. In our case $ \si_\ell $ is equivalent with the representation $ \text{ind}_P^{N_{7}} \ch_\ell $.
	For $ k\in \N^*, \ph\in C_{\ol \O}$, we have the definition of the linear
	operator  $ \si_{k,\ol\O}(\ph) $:
	\begin{eqnarray*} 
	 \si_{k,\ol \O}(\ph)&:=&\sum_{i\in\{5,6\}}\sum_{j\in \Z^*}M_
		{V^{k}_{i,j,1}}
		\circ
		\si_{
			{g^{k}_{i,j}}^{-1}\cdot p^{k}_{j,i}
		}(\ph)\circ M_{U^{k}_{i,j,1}}\\ && +\sum_{j\in \Z^*}{V^{k}_{4,j,1}}
		\circ
		\si_{
			p^{k}_{j,4}
		}(\ph)\circ M_{U^{k}_{4,j,1}}.
	\end{eqnarray*}
	and for $ a\in C^*(N_{7,2}) $ let:
	
	\begin{eqnarray*}
		\tilde\si_{k,\ol \O}(a):=\si_{k,\ol \O}(\widehat{a}_{|L(\ol \O)}).
	\end{eqnarray*}
	
\end{defn}
The proof of the next proposition is similar to that of Proposition \ref{noolO}.
\begin{prop}\label{noolO2} There is a constant $ C>0 $ such that
	for any $ \ph\in C_{\ol \O} $:
	\begin{eqnarray*}
		\noop{\si_{k,\ol\O}(\ph)}&\leq&C\sup_{\pi\in L(\ol \O)}\noop{\ph(\pi)}.
	\end{eqnarray*}
\end{prop}

\subsubsection{\bf Passing from $\GA_1$ to $\GA_0$}.
\begin{defn}
	We can  take a Schwartz-function $\eta$ in $\mathcal{S}(\R)$ such that $||\eta||=1$.
	 The function is defined by:
	\begin{enumerate}
		\item[1.]
		$\eta:=\eta_{f_5,f_6}(f_1,f_2,f_2,f_3)(s):=r_{f_5,f_6}^{\frac{1}{4}}exp(2i\pi s. (f_1,f_2))\eta(r_{f_5,f_6}^{\frac{1}{2}}(s+\frac{(f_3,f_4)}{r_{f_5,f_6}})),$\\ $ \forall (f_1,f_2,f_3,f_4)\in \R^4$ and $\forall s=(s_1,s_2)\in \R^2$.
	\end{enumerate}
	
\end{defn}
\begin{lem}
	
	 We take$\xi \in \mathcal{S}(\R)$, then
	\begin{equation*}
		\xi=\frac{1}{r_{f_5,f_6}}\int_{\R^{4}}<\xi,\eta_{f_5,f_6}(f_1,f_2,f_3,f_4)>\eta_{f_5,f_6}(f_1,f_2,f_3,f_4)df_1df_2df_3d_4.\\
	\end{equation*}
\end{lem}
\begin{proof}
	
	We use the same idea in the proof of Lemma $(2.2)$ in \cite{Lud-Tur}.
\end{proof}
\begin{defn}

	\begin{itemize}
		\item[]
		\item
		For $(f_1,f_2,f_3,f_4)\in \R^4$ and $k\in \N$ let $P_{f_5,f_6(f_1,f_2,f_3,f_4)}$ is the orthogonal projection on the subset with  one dimensional  $\C_{\eta_{f_5,f_6}}(f_1,f_2,f_3,f_4)$.
		\item  Let $h\in C_0(\R^4)$ the linear operator is given by
		\begin{eqnarray}\label{exx}
		&&\sigma_{\ell} :=\si_{f_1,f_2,f_5,f_6}(h)\\ \nonumber &&:=\frac{1}{r_{f_5,f_6}}\int_{\R^{4}} h(f_1X_1^*+f_2X_2^*+f_3X_3^*+f_4X_4^*)P_{f_5,f_6(f_1,f_2,f_3,f_4)}df_1df_2df_3df_4.
		\end{eqnarray}
	\end{itemize}
	
\end{defn}
\begin{prop}(See Proposition$(2.11)$ in \cite{Lud-Tur})
	\begin{enumerate}
		\item[1.] For any  $\ell\in  \GA_{1}$ and $h\in \mathcal{S}(\R^4)$ the integrals (\ref{exx}) converges in operator norm.
		\item[2.] The operator $\si_{\ell}(h)$ is 
 a compact and $||\si_{\ell}||_{\text{op}}\leq ||h||$.
		\item[3.] The mapping $\si_{\ell}:C_{0}(\R^4)\rightarrow  \mathcal{F}
		$ is involutive  i. e. $\si_{\ell} (h^*)=\si_{\ell}(h)^*,\\ h\in C^*(\R^4)$, such that $\si_{\ell}$ is the notation of the extension of $\si_{\ell}$ to $C_0(\R^4)$.
	\end{enumerate}
\end{prop}
\begin{thm}
	Let $a\in C^*(N_{7})$ then the functions
	\begin{enumerate}
		\item[$\bullet$] $\pi_{0}(f_1,f_2,f_3,f_4):(f_1,f_2,f_3,f_4)\rightarrow \mathcal{F}(a)(f_1,f_2,f_3,f_4)$
		
	\end{enumerate}  are contained in $C_0(\R^4)$. Let\begin{enumerate}
		\item[$\bullet$] $\overline{\gamma}=(\gamma_k=(f_5^k,f_6^k))_k$
		be a properly converging sequence in $\GA_{1}$ having its limit set $L(\overline{\gamma})=\R X_1^*+\R X_2^*+\R X_3^*+\R X_4^*$.
	\end{enumerate}Then,
	\begin{eqnarray*}
		\underset{k\rightarrow +\infty}{\lim}||\pi_0(\gamma_k)-\si_{\gamma_k}(\pi_0)||_{\text{op}}=0.
	\end{eqnarray*}
\end{thm}
\begin{proof}
	See	the proof of  Theorem$(2.12)$ in \cite{Lud-Tur}.
\end{proof}
 \begin{center}
 	\textbf{Conclusion.}
 \end{center}
We have thus established the following theorem:
\begin{thm}
	The $C^*$-algebra of the nilpotent Lie group  $N_{7}$ has NCDL.
\end{thm}


\begin{thebibliography}{10}
			\bibitem{Ell-Lud}
	F. Abdelmoula, M. Elloumi,  J. Ludwig, The $C^{*}$-algebra of the motion group $SO(n)_n\ltimes \R^n$. Bulletin des sciences mathematiques 135.2 (2011): 166-177.

\bibitem  {Cor-Gre} Corwin,L.J, Greenleaf,F.P.  Representations of  nilpotent 
Lie groups and their application. Part I. Basic theory and examples. Cambridge
Studies in Advanced Mathematics, 18. Cambridge University Press, Cambridge,    
(1990). 142 pp.
\bibitem{Dixmier} J. Dixmier, Les C*-alg\`ebres et leurs
repr\'esentations, Gauthier-Villars. (1977)
\bibitem  {Fell-1} J.M.G.Fell,A Hausdorff Topology for the Closed
Subsets of a locally Compact Non-Hausdorff
Space. Trans. Amer. Math. (1962) Soc.472-476. 
\bibitem  {Fell-2} J.M.G.Fell, Weak Containment and Induced
Representation of Groups
Space.Trans. Amer. Math.(1964) Soc.424-432. 
\bibitem{chin} M.-P. Gong, Classification of Nilpotent Lie Algebras of
Dimension $7$
(Over Algebraically Closed Fields and $\mathbb{R}$. Thesis, University of Waterloo. (1998)
	\bibitem{Gun-Lud} J.-K. G\"unther, J. Ludwig, The $C^*$-algebras of connected real two-step nilpotent Lie
groups. Revista Matem\'atica Complutense. (2016) 29(1), pp.13-57, 10.1007/s13163-015-0177-7.

\bibitem {Lep-Lud} H.~Leptin, J.~Ludwig, Unitary
representation theory
of
exponential Lie groups,
De Gruyter Expositions in Mathematics 18. (1994)
\bibitem  {Lin-Lud} J. Ludwig, Y-F. Lin, The $C^*$-algebras of
ax+b-like  groups, journal of Functional Analysis. (2010)  259. 104-130. 
\bibitem{Lud-Reg2}  J. Ludwig, H. Regeiba,  \textit{The $C^*$-algebra of some $6$-dimensional nilpotent Lie groups},
Advances in Pure and Applied Mathematics Volume $5$, issue $3$ (Aug $2014$).
\bibitem{Rej-Lud} J. Ludwig, H. Regeiba, The $C^*$-Algebra of the Heisenberg Motion Groups $\mathbb{T}^n\ltimes \mathbb{H}_{n}$. Complex Analysis and Operator Theory 13 (2019): 3943-3978
\bibitem  {Lud-Reg} J.Ludwig, H.Regeiba, $C^\ast$-algebras with norm controlled dual limits and   nilpotent Lie groups. Journal of Lie
Theory 25. (2015). no. 3, 613-655.
\bibitem  {Lud-Tur} J. Ludwig, L. Turowska, The $C^*$-algebras of the
Heisenberg group and 
thread-like Lie groups. Math. Z. 268. (2011). no. 3-4, 897-930.
\bibitem  {Nie} O.Nielsen, Unitary representations, and coadjoint orbits
of low-dimensional nilpotent Lie groups. Queen's Papers in Pure and Applied
Mathematics, 63. Queen's University, Kingston, ON. (1983). xiii+117 pp.
\bibitem{Reg.the} H. Regeibe, The $C^*$-algebra of nilpotent Lie groups of dimension $\leq6$. Thesis, Institute Ellie Carton of Lorraine
and University of Sfax, Tunisia. (2014)  
\bibitem{Regeiba}  H. Regeiba, The $C^*$-algebra of the variable Mautner group.  Banach Journal of Mathematical Analysis (2022). 16.4 . 65.
\bibitem{Hed-Jea} H. Regeiba, J. Ludwig, The $C^*$-algebra of the semi-direct product $K\ltimes A$. Monatshefte f\"ur. (2019).	
		
		
	\end{thebibliography}
\end{document}